\documentclass[11pt]{amsart}
\usepackage{amsmath,amssymb,amsthm,stmaryrd,latexsym,amsfonts,mathtools} 
\usepackage[top=30truemm,bottom=30truemm,left=25truemm,right=25truemm]{geometry}

\newtheorem{thm}{Theorem}[section]
\newtheorem{lemm}[thm]{Lemma}

\newtheorem{prop}[thm]{Proposition}

\newtheorem{rem}[thm]{Remark}
\newtheorem{cor}[thm]{Corollary}
\theoremstyle{definition}

\makeatletter
\@addtoreset{equation}{section}

\makeatother


\begin{document}

\title[Renormalization of SPDEs driven by subordinate cylindrical Brownian noises]{Renormalization of stochastic nonlinear heat and wave equations driven by subordinate cylindrical Brownian noises}
\author{Hirotatsu Nagoji}
\address{Graduate School of Science, Kyoto University, Kitashirakawa-Oiwakecho, Sakyo-ku, Kyoto 606-8502, Japan}
\email{nagoji.hirotatsu.63x@st.kyoto-u.ac.jp}
\date{}
\thanks{Statements and Declarations: This work was supported by JST SPRING, Grant Number JPMJSP2110. The author has no conflicts of interest directly relevant to the content of this article.}

\begin{abstract}
In this paper, we study the stochastic nonlinear heat equations (SNLH) and stochastic nonlinear wave equations (SNLW) on two-dimensional torus $\mathbb{T}^2 = (\mathbb{R}/2\pi\mathbb{Z})^2$ driven by a subordinate cylindrical Brownian noise, which we define by the time-derivative of a cylindrical Brownian motion subordinated to a nondecreasing c\`adl\`ag stochastic process. To construct the solution, we introduce a suitable renormalization similarly to \cite{heat} and \cite{wave}. For SNLH, we cannot expect the time-continuity for the solutions because the noise is jump-type. Moreover, due to the low time-integrability of the solutions, we could establish a local well-posedness result for SNLH only with a quadratic nonlinearity. On the other hand, for SNLW, the solutions have time-continuity and we can show the local well-posedness for general polynomial nonlinearities. Through this example, we can see that the heat case behaves worse than the wave case in the singular noise of jump-type cases.
\end{abstract}
\maketitle

\section{Introduction}
In this paper, we introduce ``subordinate cylindrical Brownian motions'' $W_L$ and consider the stochastic nonlinear heat equation
\begin{gather}
\begin{cases}
\partial_t u - \Delta u = \pm{u^k} + \partial_tW_L  \label{nlh} \\
u(0) = u_0
\end{cases}
\end{gather}
and stochastic nonlinear wave equation
\begin{gather}
\begin{cases}
\partial_t^2 u - \Delta u = \pm{u^k} + \partial_tW_L  \label{nlw} \\
(u(0),\partial_t u(0)) = (u_0,u_1)
\end{cases}
\end{gather}
on two-dimensional torus $\mathbb{T}^2 = (\mathbb{R}/2\pi\mathbb{Z})^2$
where $k \ge 2$ is an integer. We define $W_L$ by a cylindrical Brownian motion $W$ subordinated to an $\mathbb{R}_+$-valued nondecreasing c\`adla\`g stochastic process $L$. See Section 3 for more precise definition. We investigate the existence of the time-local mild solutions of these equations.
Our motivation to study the equations \eqref{nlh} and \eqref{nlw} is as follows.

Firstly, the noise is ``singular'' and jump-type: The spatial roughness of a subordinate cylindrical Brownian noise in \eqref{nlh} and \eqref{nlw} is essentially same as that of a space-time white noise, so it is a ``singular'' noise especially in the two or higher dimensional case. Stochastic nonlinear heat equations and stochastic nonlinear wave equations with additive space-time white noise (which is a ``continuous'' noise) in the 2 or 3 dimensional settings are extensively studied recent years, for example in \cite{cc18,heat,weber} and \cite{3dwave,wave,global}, especially after the invention of the theory of regularity structure \cite{rs} and the theory of paracontrolled calculus \cite{pc}. This kind of equations which are driven by very rough noises, and often involve renormalization argument, are called singular SPDEs and have attracted the attention of many researchers. Differently from these works, our choice of the noise is a jump-type noise. Through this example, we see that the heat equation behaves worse than the wave equation in the rough noise of jump-type case. More precisely, we show that we need to restrict the nonlinearity for solvability of the heat equation \eqref{nlh}, but we do not need such restrictions for the wave eqution \eqref{nlw}, see Theorems 1.1 and 1.3.  One may think that the heat case should behave more favorably because of the better smoothing property of the heat semigroup. For example, in \cite{oka}, they saw that the heat case behaves better than the wave case in the rough noise setting by considering the fractional derivative of space-time white noise. So our results suggest that this is not the case of the jump-type noise because our situation is the opposite of theirs. 

Secondly, $W_L$ is a L\'evy process which is not Gaussian in general, but renormalization argument can be applied to \eqref{nlh} and \eqref{nlw} by making use of the Gaussianity of $W$: When $L$ is a positive-valued L\'evy process, it is known that $W_L$ is also a L\'evy process. As mentioned above, there are a lot of works on the SPDEs driven by (continuous) Gaussian noises. As driving noises of SPDEs, non-Gaussian L\'evy noises, or jump noises are also considered to be natural both theoretically and in terms of applications and are actively studied by many researchers.
Especially after the publication of the monograph \cite{peslevy}, existence and uniqueness of solutions and many other important properties for SPDEs driven by jump noises are studied as an extension of the results in the case of Gaussian noises, see for example \cite{alve,brze4,mey,rud,zhu}. Also in the context of applications to physics, jump noises are considered as important examples of driving forces of SPDEs, see \cite{brze1,brze2,brze3,dede,deb,kuk1,kuk2,pes} and references therein. In the case of Gaussian noises, we can deal with nonlinear SPDEs driven by ``rougher'' noises thanks to the renormalization argument. In the renormalization process, we heavily use the Gaussianity of the noise, so similar argument can not be applied to the non-Gaussian case. In our setting, however, we can deploy a renormalization by making use of the Gaussianity of $W$, despite $W_L$ itself is not a Gaussian process in general.
We also note that our choice of the noise includes some interesting type of noises such as compound Poisson noise and stable noise, see Section 3. When $L$ is an $\alpha$-stable subordinator (i.e. an $\alpha$-stable L\'evy process with non-decreasing paths) for some $\alpha \in (0,1)$, $\partial_t W_L$ becomes a type of the stable noise. From the physical motivation, SPDEs driven by stable noises receive great attention and are studied by many researchers, see \cite{stab1,stab2,stab3} and references therein.

From these motivations, we consider the equations \eqref{nlh} and \eqref{nlw}. Our method is based on the combination of the probabilistic and analytic argument which is similar to \cite{heat} and \cite{wave}. We note that Da Prato-Debussche trick (first order expansion which they introduced in \cite{heat}) is sufficient and we do not need the theory of regularity structure nor the theory of paracontrolled calculus because we consider the equations on two-dimensional torus. 

\noindent \textbf{Acknowledgement.}
The author would like to thank Professor Seiichiro Kusuoka and Professor Yoshio Tsutsumi for helpful comments on the manuscript.

\subsection{Strategy and main result}

In the following, we briefly discuss how we approach to the equation \eqref{nlh} and \eqref{nlw} by way of renormalization procedure. Our strategy is similar to \cite{heat} and \cite{wave} but needs some modification.

First, we define the approximation operator $P_N$ by
\begin{equation}\label{aph}
P_N f := \frac{1}{2\pi} \sum_{l \in \mathbb{Z}^2 , |l| \le N} \hat{f} (l) e_l \ \ \mathrm{for} \ f \in \mathcal{D}'(\mathbb{T}^2) 
\end{equation}
where $\mathcal{D}'(\mathbb{T}^2)$ denotes the topological dual of $C^\infty(\mathbb{T}^2)$, $ e_l (x) := e^{\sqrt{- 1} l\cdot x}$, and $\hat{f} (l) = \frac{1}{2\pi} \int_{\mathbb{T}^2} f(x) e_l (x) dx $ denotes the Fourier transformation of $f$. By applying this operator to \eqref{nlh}, we consider the following regularized equation:
\begin{equation}
\mathcal{L} u_N = \pm u_N ^k + P_N \partial_t W_L  \label{anl}
\end{equation}
where $\mathcal{L} = \partial_t - \Delta$ or $\partial_t^2 - \Delta$. For simplicity, we consider the equation \eqref{anl} under the initial condition $u_N(0) = 0$ when $\mathcal{L} = \partial_t - \Delta$, and $(u_N(0),\partial_t u_N(0)) =(0,0)$ when $\mathcal{L} = \partial_t^2 - \Delta$ for a while.

Then, we deploy the Da Prato-Debussche trick. We decompose the solution $u_N$ as $u_N = X_N + v_N$ where $X_N$ is the solution of the stochastic linear equation:
\begin{gather}
\mathcal{L} X_N = P_N \partial_t W_L  \label{w}
\end{gather}
with the initial condition 0.
We define $X_N$ by the stochastic convolution
\begin{gather}
X_N(t) = \mathcal{L}^{-1} \partial_tW_L (t) = 
\begin{cases}
\displaystyle{\int^t_0 e^{(t-s)\Delta} dW_L(s)} \ \ \ \ \ \  \ \mathrm{when} \ \mathcal{L} = \partial_t - \Delta \\
\displaystyle{\int^t_0 \frac{\sin((t-s)|\nabla|)}{|\nabla|} dW_L(s)} \ \ \mathrm{when} \ \mathcal{L} = \partial_t^2 - \Delta.
\end{cases}
\end{gather}
See Section 3 for the precise definitions.
Then, $v_N$ solves the equation
\begin{equation} \label{hf}
\mathcal{L} v_N = \pm (v_N + X_N)^k = \pm \sum^k_{l = 0} \binom{k}{l} v_N ^{k - l} X_N ^l 
\end{equation}
under the initial condition 0. 
When we take the limit in $N$ to infinity, $X_N$ does not converge as functions but as distributions in $B^{- \epsilon}_{\infty,\infty} (\mathbb{T}^2)$ for any $\epsilon > 0$, see Section 4. Therefore, $X^l$ is ill-defined for $l \ge 2$. 

In view of this fact, similarly to \cite{heat} and \cite{wave}, we replace $X_N^l$ in \eqref{hf} with the ``Wick product''
\[ X_N ^{\Diamond l} (t) \coloneqq H_l (X_N (t) ; c_N (t)) \]
where $H_l$ is the Hermite polynomial i.e. $H_l$ is defined by 
\[e^{tx - \frac{1}{2}\sigma^2 t^2} = \sum^\infty_{k = 0} \frac{t^k}{k!} H_k (x;\sigma^2) \]
and  $c_N (t)$ is a renormalization constant defined by 
\begin{gather}
c_N (t) \coloneqq 
\begin{cases}
\displaystyle{\frac{1}{(2\pi)}\sum_{l \in \mathbb{Z}^2 , |l| \le N} \int^t_0 e^{2(s-t)|l|^2} dL(s)} \ \ \ \mathrm{when} \ \mathcal{L} = \partial_t - \Delta \\
\displaystyle{\frac{1}{(2\pi)^2} \sum_{l \in \mathbb{Z}^2 , |l| \le N} \int^t_0 \frac{\sin^2 ((t - s)|l|)}{|l|^2} dL(s)} \ \ \ \mathrm{when} \ \mathcal{L} = \partial^2_t - \Delta. \label{gg}
\end{cases}
\end{gather}
Note that $c_N (t)$ depends on $t \in \mathbb{R}_+$ and $\omega \in \Omega$, however, is independent of $x \in \mathbb{T}^2$. Moreover, $c_N$ is an $L$-measurable $\mathbb{R}_+$-valued stochastic process independent of $W$.
Then, we can show that $X^{\Diamond l}_N$ converges to some $X^{\Diamond l}$, see Section 4.
The renormalized equation, given by the replacement $X_N^l$ in \eqref{hf} with $X_N ^{\Diamond l}$ is 
\begin{align}
\mathcal{L} \hat{v}_N &= \pm \sum^k_{l = 0} \binom{k}{l}\hat{v}_N ^{k - l} X_N ^{\Diamond l} = \pm H_k ( \hat{v} + X_N ; c_N) \label{lalala}
\end{align}
where we use the property of the Hermite polynomials
\[ H_k ( a + b ; c) = \sum^k_{l = 0} \binom{k}{l}a^{k - l} H_l (b;c). \]
Therefore, $\hat{u}_N = X_N + \hat{v}_N$ solves the equation 
\begin{align}
\mathcal{L} \hat{u}_N &= \pm H_k ( \hat{u}_N; c_N) + \partial_tW_L. \label{usol}
\end{align}
Then, we can show that $\hat{u}_N$ converges to $\hat{u} = X + \hat{v}$ where $\hat{v}$ is the solution of the equation
\begin{equation} \label{ff}
\mathcal{L} \hat{v} = \pm \sum^k_{l = 0} \binom{k}{l}\hat{v} ^{k - l} X ^{\Diamond l} 
\end{equation}
and we define by $\hat{u}$ the solution of the renormalized equation
\begin{align}   
\mathcal{L}\hat{u} = \pm \hat{u}^{\Diamond k} + \partial_tW_L .
\end{align}

Now, we state the main result. In the following, $\Phi \coloneqq (\partial_t - \Delta)^{-1} \partial_tW_L (t)$ denotes the solution of \eqref{w} with the initial condition $\Phi(0)=0$ and $\Psi \coloneqq  (\partial^2_t - \Delta)^{-1} \partial_tW_L (t)$ denotes the solution of \eqref{w} with the initial condition $(\Psi(0),\partial_t \Psi(0))=(0,0)$.
\begin{thm}\label{adtta}
Let $k=2$, $0<\epsilon<\frac{1}{2}$. Then, there exists a unique local-in-time mild solution of the renormalized stochastic nonlinear heat equation 
\begin{gather}
\begin{cases}
\partial_t u - \Delta u = \pm u^{\Diamond k} + \partial_tW_L \label{hgr}\\
u(0) = u_0  
\end{cases}
\end{gather}
in 
\[\Phi + L^\gamma (\left[0,T\right];B^{2/\gamma-\delta}_{\infty,\infty}(\mathbb{T}^2) ) \cap C([0,T];B^{-\delta}_{\infty,\infty}(\mathbb{T}^2)) \]
$\mathbb{P}$-almost-surely for any $\frac{2}{1-\epsilon}<\gamma<\frac{2}{\epsilon}$, $0<\delta<\frac{2}{\gamma}-\epsilon$ and initial conditions $u_0 \in B^{2/\gamma-\delta}_{\infty,\infty}(\mathbb{T}^2)$. More precisely, there exists a random time $T(\omega)>0$ such that the equation
\begin{equation}
v(t) = e^{t\Delta}u_0 \pm \sum^2_{l = 0} \binom{2}{l}\int^t_0 e^{(t-s)\Delta} v ^{2 - l}(s) \Phi ^{\Diamond l}(s) ds
\end{equation}
has a unique solution in $ L^\gamma (\left[0,T\right];B^{2/\gamma-\delta}_{\infty,\infty}(\mathbb{T}^2) ) \cap C([0,T];B^{-\delta}_{\infty,\infty}(\mathbb{T}^2))$ $\mathbb{P}$-almost-surely, where $\Phi^{\Diamond l}$ are the renormalized powers defined in Proposition \textup{\ref{rere}}.
\end{thm}
\begin{rem}
When $L(t) = t$ i.e. $W_L=W$, the local well-posedness of the equation for all $k\ge2$ is proved in \cite{heat}.
However, in our setting, the lack of time-integrability of the renormalized power $\Phi^{\Diamond k}$ causes a problem and we have not been able to solve the equation \eqref{hgr} for $k\ge 3$. See Section \textup{4} for more details. 
\end{rem}

For the renormalized stochastic wave equation \eqref{hgi}, there holds the following result.
\begin{thm}\label{adta}
Let $k\ge2$. Then, there exists a unique local-in-time mild solution of the renormalized stochastic nonlinear wave equation
\begin{gather}
\begin{cases}
\partial_t^2 u - \Delta u = \pm u^{\Diamond k} + \partial_tW_L  \label{hgi} \\
(u(0),\partial_t u(0)) = (u_0,u_1) 
\end{cases}
\end{gather}
in  
\[ \Psi +  C([0,T];H^{1-\epsilon}(\mathbb{T}^2)) \cap C^1([0,T];H^{-\epsilon}(\mathbb{T}^2)) \]
$\mathbb{P}$-almost-surely for any $0<\epsilon<\frac{1}{2(k-1)}$ and initial conditions $(u_0,u_1) \in H^{1-\epsilon}(\mathbb{T}^2)\times H^{-\epsilon}(\mathbb{T}^2)$. More precisely, there exists a random time $T(\omega)>0$ such that the equation
\begin{equation}\label{a}
v(t) = \cos(t|\nabla|)u_0 + \frac{\sin((t-s)|\nabla|)}{|\nabla|}u_1 \pm \sum^k_{l = 0} \binom{k}{l}\int^t_0 \frac{\sin((t-s)|\nabla|)}{|\nabla|} v ^{k - l}(s) \Psi ^{\Diamond l}(s) ds
\end{equation} 
has a unique solution in $C([0,T];H^{1-\epsilon}(\mathbb{T}^2)) \cap C^1([0,T];H^{-\epsilon}(\mathbb{T}^2))$ $\mathbb{P}$-almost-surely, where $\Psi^{\Diamond k}$ are the renormalized powers defined in Proposition \textup{\ref{yy}}.
\end{thm}
\begin{rem}
Differently from the equation \eqref{hgr}, we can solve the equation \eqref{hgi} for all $k\ge2$. This is due to the difference of time-integrability between $\Phi^{\Diamond k}$ and $\Psi^{\Diamond k}$. See Propositions \textup{\ref{rere}} and \textup{\ref{yy}} for more details.
\end{rem}
\begin{rem}\label{scifi}
In view of Proposition \textup{\ref{yy}}, we can apply the argument in \cite{wave} to construct the solution of \eqref{a}. As a consequence,
the local well-posedness holds in the space 
\[C([0,T];H^s(\mathbb{T}^2)) \cap C^1([0,T];H^{s-1}(\mathbb{T}^2))\cap L^q([0,T];L^r(\mathbb{T}^2))\] for $0\ll s<1$ and $s$-wave-admissible pair $(q,r)$. This result is proved by the Strichartz estimates for wave equations. See \cite{wave} for the details. 
\end{rem}
\begin{rem}
For the equation with a cubic nonlinearity 
\begin{equation}\label{msn}
\partial_t^2 u - \Delta u = - u^{\Diamond 3} + \partial_tW_L,
\end{equation}
it is expected that we can show the existence of the global solution. When $L(t) = t$ i.e. $W_L=W$, the global well-posedness of the equation is proved in \cite{global} by the I-method. Once we control the behaviors of $\Psi^{\Diamond l}$ by the probabilistic argument, their method is based on the pathwise deterministic argument. It seems that their proof can be applied to the equation \eqref{msn}.
\end{rem}

\begin{rem}
By a similar argument, the statements of Theorems \ref{adtta} and \ref{adta} are easily extended to the case of the equations with general polynomial nonlinearities
\begin{equation}
\partial_t u - \Delta u = \sum^2_{l=0} a_l u^{\Diamond l} + \partial_t W_L
\end{equation}
and
\begin{equation}
\partial_t^2 u - \Delta u = \sum^k_{l=0} a_l u^{\Diamond l} + \partial_t W_L
\end{equation}
where $a_l \in \mathbb{R}$ and we interpret as $u^{\Diamond 0} = 1$.
\end{rem}

\begin{rem}
To construct the renormalized powers $\Phi^{\Diamond l}$ and $\Psi^{\Diamond l}$, we used the approximation operator $P_N$ defined in \eqref{aph}. More precisely, we defined as
\[ \Phi^{\Diamond l} \coloneqq \lim_{N\rightarrow \infty}H_l \left( P_N \Phi ;c_N(t)\right) \]
where $c_N(t)$ is the renormalization constant defined in \eqref{gg}. Actually, one may use a more general approximation procedure. Let $\psi$ be a radially symmetric Borel function on $\mathbb{R}^2$ with $0\le\psi\le 1$,
\[ \sup_{x\in\mathbb{R}^2\backslash \{0\}} \frac{|\psi(x) - 1|}{|x|^\theta} < +\infty \ \ \ \mbox{for some}\ \theta >0,\]
and 
\[\sup_{x\in\mathbb{R}^2} |x|^\eta |\psi(x)| <+\infty \ \ \ \mbox{for some}\ \eta>0. \]
We define an approximation operator $Q_N$ on $\mathcal{D}'(\mathbb{T}^2)$ by
\[ Q_N f \coloneqq \psi\left(\frac{\nabla}{N}\right) \coloneqq \frac{1}{2\pi}\sum_{l\in \mathbb{Z}^2}\psi \left(\frac{l}{N}\right)\hat{f}(l)e_l. \]
This kind of approximation operator is considered in \cite{exp}. Note that $P_N$ in \eqref{aph} is an example of $Q_N$.
Then, similarly, we can show the existence of the limit  
\[ \Phi^{\Diamond l} \coloneqq \lim_{N\rightarrow \infty}H_l \left( Q_N \Phi ;c^Q_N(t)\right) \]
where $c^Q_N$ depends on $Q_N$. One can prove that $\Phi^{\Diamond l}$ does not depend on the choise of $Q_N$.
\end{rem}

The organization of the present paper is as follows. In Section 2, we briefly recall the basics of Besov spaces, Sobolev spaces, and  Gaussian random variables. In Section 3, we introduce the subordinate cylindrical Brownian motion $W_L$ and define the stochastic integral with respect to $W_L$. In Section 4, we prove the existence of the renormalized powers (Wick powers) constructed from $W_L$ and study their properties such as the spacial regularity and time-integrability. We deploy the renormalization in this section. In Section 4.3, we also consider the renormalized powers of the stationary Ornstein-Uhlenbeck processes. In Section 5, we prove Theorem \ref{adtta} and Theorem \ref{adta}, which state the local well-posedness of the renormalized equations \eqref{hgr} and \eqref{hgi}.

\subsection{Notations}
We use the following notations:
\begin{enumerate}
\item
$\mathbb{R}_+ = \{ t\in \mathbb{R};t \ge 0\},\quad \mathbb{R}_- = \{ t\in \mathbb{R};t \le 0\}$ 
\item 
$\mathcal{D}'(\mathbb{T}^d)$ denotes the topological dual of $C^\infty(\mathbb{T}^d)$.
\item
We write $a\lesssim b$ if there holds $a\le C b$ for some constant $C$ independent of the variables under consideration. When $C$ depends on the variable $x$ and we want to emphasize it, we write $a \lesssim_x b$. We also write $a \simeq b$ if $a\lesssim b$ and $a \gtrsim b$.
\item
For $k, l \in \mathbb{Z}^d$,
\begin{gather}
\delta_{k,l} \coloneqq 
\begin{cases}
\displaystyle{1} \ \ \ \mathrm{when} \ k=l \\
\displaystyle{0} \ \ \ \mathrm{when} \ k\neq l \notag
\end{cases}
\end{gather}
denotes the Kronecker delta.
\end{enumerate}

\section{Preliminary}

\subsection{Function spaces}
For a distribution $f \in \mathcal{D}' (\mathbb{T}^d)$, we define the Fourier transformation by
\[ \mathcal{F} f (l) = \hat{f} (l) = \langle f, e_{-l} \rangle \ \ \ \mathrm{for}\  l \in \mathbb{Z}^d \]
where $\{e_l (x) \}_{l \in \mathbb{Z}^2} = \{e^{\sqrt{-1}l\cdot x}\}_{l \in \mathbb{Z}^2}$ denotes the Fourier basis of $L^2 (\mathbb{T}^2)$ and $\langle \ , \ \rangle$ denotes the pairing defined as the extension of the normalized $L^2$-inner product
\[ \langle f, g \rangle \coloneqq \frac{1}{(2\pi)^{\frac{d}{2}}} \int_{\mathbb{T}^d} f(x)g(x) dx .\]

Sobolev spaces (Bessel potential spaces) $W^{\alpha, p}(\mathbb{T}^d)$ for $\alpha \in \mathbb{R}, 1 \le p \le \infty$ are defined as the space of all $f \in \mathcal{D}' (\mathbb{T}^d)$ with 
\[ \| f \|_{W^{\alpha, p}(\mathbb{T}^d)} \coloneqq \left\| \langle \nabla \rangle^{\alpha} f \right\|_{L^p (\mathbb{T}^d)} = \| \sum_{l \in \mathbb{Z}^d} \left( 1 + |l|^2 \right)^\frac{\alpha}{2} \hat{f}(l)e_l \|_{L^p(\mathbb{T}^d)} < \infty. \]

Next, in order to define Besov spaces, we introduce the Littlewood-Paley blocks. Let $\chi, \rho \in C_c ^\infty (\mathbb{R}^d)$ be $\mathbb{R}_+$-valued functions such that
\begin{enumerate}
\item
$\mathrm{supp}(\chi) \subset{B(4)},\  \mathrm{supp}(\rho) \subset{B(4)\backslash B(1)}$,
\item
$\chi (x) + \sum_{i=0}^\infty \rho (2^{-i}x) = 1$ for any $x \in \mathbb{R}^2$.
\end{enumerate}
Such a pair $(\chi,\rho)$ indeed exists, see \cite[Section 2.2]{bcd} for the proof.
For the convenience, we write
\[ \rho_{-1} = \chi, \ \rho_{j} = \rho(2^{-j}\cdot) \ \ \ \mathrm{for} \ j \ge 0 \]
and define 
\[ \Delta_m f \coloneqq \rho_m(\nabla) f = \frac{1}{2\pi} \sum_{l \in \mathbb{Z}^d} \rho_m(l) \hat{f}(l) e_l \ \ \mathrm{for} \ f \in \mathcal{D}'(\mathbb{T}^d).\]
Besov spaces $B^\alpha_{p,q}(\mathbb{T}^d) $ for $1 \le p,q \le \infty, \alpha \in \mathbb{R}$ are defined as the space of all $f \in \mathcal{D}'(\mathbb{T}^d)$ with the finite Besov norm
\[ \| f \|_{B^\alpha_{p,q} (\mathbb{T}^d)} \coloneqq \| 2^{m\alpha} \Delta_m f \|_{l^q(L^p (\mathbb{T}^d))} < \infty .\]

\begin{lemm}[Product estimates]\label{pro}
For $\alpha, \beta \in \mathbb{R}$ with $\alpha + \beta >0$, 
\[ \|fg\|_{B^{\alpha \wedge \beta}_{\infty,\infty}(\mathbb{T}^d)} \lesssim \|f\|_{B^\alpha_{\infty,\infty}(\mathbb{T}^d)} \|g\|_{B^\beta_{\infty,\infty}(\mathbb{T}^d)}. \]  
\end{lemm}
\begin{proof}
This follows from the paraproduct estimates. See \cite[Lemma 2.1]{pc}, for example.
\end{proof}

\begin{lemm}[Besov embeddings] \label{emb}
Let $1 \le p_1 \le p_2 \le \infty, 1\le q_1 \le q_2 \le \infty$ and $\alpha \in \mathbb{R}$. Then, $ B^\alpha_{p_1,q_1}(\mathbb{T}^d)$ is continuously embedded in $B^{\alpha - d(\frac{1}{p_1}-\frac{1}{p_2})}_{p_2,q_2}(\mathbb{T}^d)$.
\end{lemm}
\begin{proof}
See \cite[Proposition 2.71]{bcd} for the proof on $\mathbb{R}^d$. The proof on $\mathbb{T}^d$ is similar. 
\end{proof}

\subsection{Probabilistic tools}
Let $H_k$ be the $k$th Hermite polynomial
\[ H_k (x) = \sum^{[\frac{k}{2}]}_{i=0} (-1)^i \frac{k!}{2^i i! (k-2i)!} x^{k-2i} .\]
The first Hermite polynomials are given by 
\[ H_0 (x) = 1,\  H_1(x) = x,\ H_2(x) = x^2 -1,\ H_3(x) = x^3 -3x. \]
For $\sigma^2 >0$, the generalized Hermite polynomial $H_k (x;\sigma^2)$ is defined by
\[ H_k (x;\sigma^2) = \sigma^k H_k \left(\frac{x}{\sigma}\right). \]
The Hermite polynomials satisfies
\[ e^{tx-\frac{1}{2}\sigma^2 t^2} = \sum^\infty_{k = 0} \frac{t^k}{k!} H_k (x;\sigma^2).\]

\begin{lemm} \label{gauss}
For $\xi_1 \sim N(0,\sigma_1^2)$ and $\xi_2 \sim N(0,\sigma_2^2)$, there holds
\[ \mathbb{E} [ H_k (\xi_1, \sigma_1^2) H_m (\xi_2, \sigma_2^2) ] = k! \delta_{km} \mathbb{E}[\xi_1 \xi_2]^k .\]
\end{lemm}
\begin{proof}
See \cite[Lemma 1.1.1]{nua}.
\end{proof}

\section{Subordinate cylindrical Brownian motion}
\subsection{Definition}
Let $(\Omega_1, \mathcal{F}_1, \mathbb{P}_1), (\Omega_2, \mathcal{F}_2, \mathbb{P}_2)$ be two probability spaces and consider the product space $(\Omega, \mathcal{F}, \mathbb{P}) \coloneqq (\Omega_1 \times \Omega_2, \mathcal{F}_1 \otimes \mathcal{F}_2, \mathbb{P}_1 \otimes \mathbb{P}_2)$. Let $W$ be a cylindrical Brownian motion on $L^2(\mathbb{T}^2)$ defined on the probability space $(\Omega_1,\mathcal{F}_1,\mathbb{P}_1)$, formally expressed by
\[ W(t) =  \frac{1}{2\pi} \sum_{l \in \mathbb{Z}^2} \beta^l (t) e_l \]
where $(\beta^l)_{l\in \mathbb{Z}^2}$ is an independent sequence of $\mathbb{C}$-valued standard Brownian motions conditioned with $\bar{\beta^l} = \beta^{-l}$. We normalize $\beta^l$ and assume $\mathrm{var}(\beta^l(t)) = t$. We also consider an $\mathbb{R}_+$-valued stochastic process $L$ which satisfies $L(0) = 0$ and has non-decreasing and c\`adl\`ag sample paths defined on the another probability space $(\Omega_2,\mathcal{F}_2,\mathbb{P}_2)$. By the construction, $(\beta^l)_{l \in \mathbb{Z}^2}$ and $L$ are mutually independent. 
Then, we define the subordinate cylindrical Brownian motion on $L^2 (\mathbb{T}^2)$ by 
\[ W_L (t) \coloneqq \frac{1}{2\pi} \sum_{l \in \mathbb{Z}^2} \beta^l_L (t) e_l \]
where we write $\beta^l_L(t)= \beta^l(L(t))$.
Typical example of a subordinator $L$ is a L\'evy subordinator i.e. a L\'evy process with non-decreasing sample paths such as Poisson processes. See \cite[Section 1.3.2]{levy} for more examples of L\'evy subordinators. In this case, $W_L$ is a $\mathcal{D}'(\mathbb{T}^2)$-valued L\'evy process.

\subsection{Stochastic integral}
In this subsection, we give precise definitions to the solutions of the linear stochastic heat and wave equations driven by subordinate cylindrical Brownian noises:
\begin{gather}
\begin{cases} \label{he}
\partial_t  \Phi - \Delta \Phi =  \partial_t{W_L} \\
\Phi (0) =  \phi_0, 
\end{cases} \\
\begin{cases} \label{wa}
\partial_t ^2 \Psi - \Delta \Psi = \partial_t{W_L} \\
(\Psi(0) , \partial_t \Psi(0) ) = (\psi_0, \psi_1).  
\end{cases}
\end{gather}

First, we consider the stochastic linear heat equation \eqref{he}. By the Duhamel principle, it is natural to define the solution $\Phi$ by
\begin{align*}
\Phi(t) &= e^{t\Delta}\phi_0 + \int^t_0 e^{(t-s)\Delta} dW_L (s) = e^{t\Delta}\phi_0 + \frac{1}{2\pi} \sum_{l \in \mathbb{Z}^2} \int^t_0 e^{(s-t)|l|^2} d\beta^l_L (s) e_l.
\end{align*}
So we need to define the (one-dimensional) stochastic integrals 
\begin{equation}\label{one}
\int^t_0 e^{(s-t)|l|^2} d\beta^l_L (s)
\end{equation}
and then prove the convergence of the infinite sum 
\begin{equation}\label{inf}
\frac{1}{2\pi} \sum_{l \in \mathbb{Z}^2} \int^t_0 e^{(s-t)|l|^2} d\beta^l_L (s) e_l.
\end{equation}
When $L$ is a L\'evy process, the integral $\int^t_0 e^{(s-t)|l|^2} d\beta^l_L (s)$ can be defined by the usual stochastic integral theory since $\{ \beta^l_L\}$ are semimartingales.
For general $L$, although $\{\beta^l_L\}$ are not always semimartingales, we can define the stochastic integral with respect to $\beta^l_L$ pathwisely by the Young integral thanks to the following lemma.
In the following, we use the notations in Section 6.1. Especially, $\|\cdot \|_{p,[0,T]}$ denotes the $p$-variation norm defined in \eqref{vvv} and $P[0,T]$ denotes the set of partitions of $[0,T]$.
\begin{lemm} \label{vari}
Let $B$ be a Brownian motion with continuous paths. Then, the sample paths of $B_L$ have the finite $(2+\epsilon)$-variation for any $\epsilon>0$.
Moreover, there holds
\[ \| B_L \|_{2+\epsilon, [0,T]} \le \left( \sup_{s,t \in [0,L(T)]} \frac{|B(t) - B(s)|}{|t-s|^\frac{1}{2+\epsilon}}\right) L(T)^{\frac{1}{2+\epsilon}} \]
for any $T>0$.
\end{lemm} 

\begin{proof}
We can calculate that
\begin{align*}
\| B_L \|_{2+\epsilon, [0,T]}^{2+\epsilon} &= \sup_{D\in P[0,T]} \sum^N_{i=1} \left| B_L(t_i) - B_L(t_{i-1}) \right|^{2+\epsilon} \\
&\le \left( \sup_{s,t \in [0,T]} \frac{|B_L(t) - B_L(s)|^{2+\epsilon}}{|L(t)-L(s)|}\right) \sup_{D\in P[0,T]} \sum^N_{i=1} |L(t_i) -L(t_{i-1})| \\
&\le \left( \sup_{s,t \in [0,L(T)]} \frac{|B(t) - B(s)|^{2+\epsilon}}{|t-s|}\right) L(T) < \infty
\end{align*}
where we use the $\frac{1}{2+\epsilon}$-H\"{o}lder continuity of the sample paths of the Brownian motion $B$.
\end{proof}

Because $e^{(\cdot - t)|l|^2}$ has the bounded variation for each $l \in \mathbb{Z}^2$ and there holds $\frac{1}{2+\epsilon} + \frac{1}{1} > 1$, the integral \eqref{one} is well-defined as a Young integral by Lemma\ \ref{vari}. As a consequent, the finite sum 
$\frac{1}{2\pi} \sum_{|l| \le N} \int^t_0 e^{(s-t)|l|^2}d\beta^l_L (t) e_l $
is well-defined for any $N\in \mathbb{Z}_{>0}$. Then, we can prove the convergence of 
$\left( \frac{1}{2\pi} \sum_{|l| \le N} \int^t_0 e^{(s-t)|l|^2} d\beta^l_L (t) e_l \right)_{N\in \mathbb{Z}_{>0}}$
as $N \rightarrow \infty$ and we define the infinite sum \eqref{inf} by this limit, see Section 4 below.

Next, we consider the stochastic linear wave equation \eqref{wa}. We define the solution of \eqref{wa} by the stochastic convolution
\begin{align*}
\Psi (t) &= \cos(t|\nabla|) \psi_0 + \frac{\sin(t|\nabla|)}{|\nabla|} \psi_1 + \int^t_0 \frac{\sin((t - s)|\nabla|)}{|\nabla|} dW_L(s)  \\
&:= \cos(t|\nabla|) \psi_0 + \frac{\sin(t|\nabla|)}{|\nabla|} \psi_1 + \frac{1}{2\pi} \sum_{l \in \mathbb{Z}^2} \int^t_0 \frac{\sin((t - s)|l|)}{|l|} d\beta^l_{L}(s) e_l 
\end{align*}
where we interpret as $\frac{\sin ((t-s)|0|)}{|0|}= t-s$.
Similarly to the case of the heat equation, this expression turns out to be well-defined. Indeed, we can define $\frac{1}{2\pi} \sum_{|l|\le N} \int^t_0 \frac{\sin((t - s)|l|)}{|l|} d\beta^l_{L}(s) e_l$ for any $N \in \mathbb{Z}_{>0}$ thanks to Lemma \ref{vari}, and we can show the convergence as $N \rightarrow \infty$.

\section{Renormalization}
\subsection{Lemmas}

In this subsection, we give some lemmas which we need in the proof of the convergence of the renormalized powers 
\begin{equation}
\Phi_N^{\Diamond k}(t) = H_k (\Phi_N(t), c_N^H) 
\end{equation}
\begin{equation} \label{qq}
\Psi_N^{\Diamond k}(t) = H_k (\Psi_N(t), c_N^W) 
\end{equation}
where 
\[ \Phi_N (t) = \frac{1}{2\pi} \sum_{|l| \le N} \int^t_0 e^{(s-t)|l|^2} d\beta^l_L (s) e_l \]
and
\[ \Psi_N (t) = \frac{1}{2\pi} \sum_{|l| \le N} \int^t_0  \frac{\sin((t - s)|l|)}{|l|} d\beta^l_L (s) e_l \]
are the solutions of the regularized linear equations
\begin{gather}
\begin{cases} \label{rehe}
\partial_t  \Phi_N - \Delta \Phi_N =  P_N\partial_t{W_L} \\
\Phi_N (0) =  0 
\end{cases} \\
\begin{cases} \label{rehe2}
\partial_t ^2 \Psi_N - \Delta \Psi_N = P_N\partial_t{W_L} \\
(\Psi_N(0) , \partial_t \Psi_N(0) ) = (0, 0)  
\end{cases}
\end{gather}
respectively and $c_N^H(t), c_N^W(t)$ are given by
\[c^H_N (t) = \frac{1}{2\pi} \sum_{|l| \le N} \int^t_0 e^{2(s-t)|l|^2} dL (s), \ \ \ c^W_N (t) = \frac{1}{2\pi} \sum_{|l| \le N} \int^t_0  \frac{\sin^2((t - s)|l|)}{|l|^2} dL (s). \]

In the following, we write the expectations with respect to the probability measures $\mathbb{P}_i, i = 1, 2$ by
\[ \mathbb{E}^{\mathbb{P}_i} [ X ] = \int_{\Omega_i} X(\omega) d\omega .\]
First we prove the following general lemma. 
\begin{lemm}\label{lemm1}
Let $f, g \in V^{2-\epsilon}[0,T] \cap C[0,T]$ for some $\epsilon > 0$ and $T>0$, where  the spaces $V^p$ are defined in \eqref{revise}. Then, for any fixed $\omega_2 \in \Omega_2$, $\int^T_0 f(t) d\beta^l_L (t)$ is a Gaussian random variable on $\Omega_1$ and there holds
\[ \mathbb{E}^{\mathbb{P}_1} \left[ \int^T_0 f(t) d\beta^l_L (t) \int^T_0 g(t) d\beta^{l'}_L (t) \right] 
= \delta_{l, -l'} \int^T_0 f(t)g(t) dL(t) .\]
\end{lemm}
\begin{proof}
Fix $\omega_2 \in \Omega_2$. We use the notions in Section 6.1. For $D\in P[0,T]$, $I(f,\beta^l_L,D)$ is a Gaussian random variable and by Lemma \ref{young} and Lemma \ref{vari}, it converges to $\int^T_0 f(t) d\beta^l_L (t)$ as $|D|\rightarrow 0$ for any $\omega_1\in\Omega_1$. Therefore, $\int^T_0 f(t) d\beta^l_L (t)$ is also a Gaussian random variable on $\Omega_1$.
Moreover, we can calculate that
\begin{align}
&\mathbb{E}^{\mathbb{P}_1} \left[ I\left(f,\beta^l_L,D \right) I\left(g,\beta^{l'}_L,D \right) \right] \label{tala} \\
&= \mathbb{E}^{\mathbb{P}_1} \left[ \left( \sum^N_{i=1} f(t_{i-1}) \left( \beta^l_L(t_i) - \beta^l_L(t_{i-1}) \right) \right) \left( \sum^N_{i=1} g(t_{i-1}) \left( \beta^{l'}_L(t_i) - \beta^{l'}_L(t_{i-1}) \right) \right) \right] \notag \\
&= \delta_{l,-l'} \sum^N_{i=1} f(t_{i-1}) g(t_{i-1}) \mathbb{E}^{\mathbb{P}_1} \left[ \left| \beta^{l}_L(t_i) - \beta^{l}_L(t_{i-1}) \right|^2 \right]  \notag \\
&= \delta_{l,-l'} \sum^N_{i=1} f(t_{i-1}) g(t_{i-1})   \left( L(t_i) - L(t_{i-1}) \right) \notag
\end{align}
where we use the fact that Brownian motions have the independent and stationary increments. The last line of \eqref{tala} can be dominated by
$ \|f\|_{C[0,T]}\|g\|_{C[0,T]}L(T) $
uniformly in $D\in P[0,T]$. Because $I(f,\beta^l_L,D)$ and $I(g,\beta^l_L,D)$ are Gaussian random variables, this fact yields
\[ \sup_{D\in P[0,T]} \mathbb{E}^{\mathbb{P}_1} \left[ \left| I\left(f,\beta^l_L,D \right) I\left(g,\beta^{l'}_L,D \right) \right|^p \right] < + \infty \]
for any $p\ge1$. In particular, $\{ I\left(f,\beta^l_L,D \right) I(g,\beta^{l'}_L,D ) \}_{D\in P[0,T]}$ is uniformly integrable. Thus, by taking the limit $|D| \rightarrow 0$ in \eqref{tala}, we obtain the desired result.
\end{proof}
By Lemma \ref{lemm1}, we can check that
\[ c_N^H (t) = \mathbb{E}^{\mathbb{P}_1} [ \Phi_N (t)^2 ] \]
and
\[ c_N^W (t) = \mathbb{E}^{\mathbb{P}_1} [ \Psi_N (t)^2 ] .\]

Next, we give a formula for the $\mathbb{P}_1$-expectations of $\Phi_N^{\Diamond k}$ and $\Psi_N^{\Diamond k}$.
\begin{lemm} \label{lemm2}
For any $s,t \in \mathbb{R}_+$, $k,m \in \mathbb{N}$ and $x,y \in \mathbb{T}^2$, there holds
\[ \mathbb{E}^{\mathbb{P}_1} \left[ \Phi_N^{\Diamond k} (s, x) \Phi_N^{\Diamond m} (t, y)  \right] = \delta_{k, m} k! \mathbb{E}^{\mathbb{P}_1} \left[ \Phi_{N} (s, x) \Phi_{N} (t, y) \right]^k \] and
\[ \mathbb{E}^{\mathbb{P}_1} \left[ \Psi_N^{\Diamond k} (s, x) \Psi_N^{\Diamond m} (t, y)  \right] = \delta_{k, m} k! \mathbb{E}^{\mathbb{P}_1} \left[ \Psi_{N} (s, x) \Psi_{N} (t, y) \right]^k .\]
\end{lemm}
\begin{proof}
By the definitions, $\Phi_N$ and $\Psi_N$ are $\mathbb{R}$-valued centered Gaussian random variables with respect to $\omega_1 \in \Omega_1$ for any $(s,x)\in \mathbb{R}_+ \times \mathbb{T}^2$ and fixed $\omega_2 \in \Omega_2$.
 Moreover, there holds $c_N^H (t) = \mathbb{E}^{\mathbb{P}_1} [ \Phi_N (t,x)^2 ]$
and $c_N^W (t) = \mathbb{E}^{\mathbb{P}_1} [ \Psi_N (t,x)^2 ]$ as explained above. Therefore, by Lemma \ref{gauss}, the statement follows immediately.
\end{proof}

\subsection{Convergence of the renormalized powers}
By the formulas which we derived above, we can show the convergence of $\Phi_N^{\Diamond k}$ and $\Psi_N^{\Diamond k}$ as $N \rightarrow \infty$.

First, we consider $\Phi_N^{\Diamond k}$, the renormalized powers of the solution of the regularized stochastic linear heat equation \eqref{rehe}.
Note that we cannot expect time-continuity for $\Phi^{\Diamond k}$ since we are dealing with jump-type noises. However, we can prove the convergence on $L^p$-spaces with respect to time variable $t$.
\begin{prop} \label{rere}
Fix $T > 0$ and $k \in \mathbb{Z}_{>0}$. 
For any $0<\epsilon<\frac{1}{k},\ \alpha<-\epsilon k,\ 0<\gamma < \frac{2}{(1-\epsilon)k}$, $p\ge 1$, and fixed $\omega_2 \in \Omega_2$, $\Phi_N^{ \Diamond k} (\omega_1, \omega_2)$ converges in $L^p \left( \Omega_1; L^\gamma ([0,T];B^\alpha_{\infty, \infty}(\mathbb{T}^2)) \right)$ and $\mathbb{P}_1$-almost surely.
In particular, $\Phi_N^{\Diamond k}$ converges in $L^\gamma ([0,T];B^\alpha_{\infty, \infty}(\mathbb{T}^2))$ $\mathbb{P}$-almost surely.
\end{prop}

\begin{rem}
When $0<\gamma <1$, $L^\gamma$ is not a norm space, however, it is a complete metric space.
\end{rem}

\begin{rem}
One can show that $\Phi_N$ converges to some time-discontinuous process $\Phi$ as $N\rightarrow\infty$ when $L$ has jumps, see Section 6.2.
\end{rem}

\begin{rem}
When $L(\cdot) \in C^1 (\mathbb{R}_+) $ a.s., it is straightforward to prove that $\Phi_N^{\Diamond k}$ converges in the space $ C([0,T];B^{-\epsilon}_{\infty,\infty} (\mathbb{T}^2))$ for any $\epsilon>0$. For example, it is well-known that $\Phi_N^{\Diamond k}$ converges in $C([0,T];B^{-\epsilon}_{\infty,\infty} (\mathbb{T}^2))$ if $L(t) = t$.
\end{rem}

\begin{proof}
Fix $\omega_2 \in \Omega_2$ and let $1 \le p < \infty$. Note that $\Phi_N (\omega_1,\omega_2)$ is a Gaussian random variable with respect to $\omega_1 \in \Omega_1$.
For $M> N \ge1$,
\begin{align*}
\left\| \int^T_0  \| \Phi^{\Diamond k}_{N}(t) - \Phi^{\Diamond k}_{M}(t) \|^\gamma_{B^\alpha_{2p,2p}(\mathbb{T}^2)} dt \right\|_{L^p(\Omega_1)} &\le \int^T_0 \left\|  \| \Phi^{\Diamond k}_{N}(t) - \Phi^{\Diamond k}_{M}(t) \|^\gamma_{B^\alpha_{2p,2p}(\mathbb{T}^2)} \right\|_{L^p(\Omega_1)} dt \\
&= \int^T_0 \mathbb{E}^{\mathbb{P}_1} \left[  \| \Phi^{\Diamond k}_{N}(t) - \Phi^{\Diamond k}_{M}(t) \|^{p\gamma}_{B^\alpha_{2p,2p}(\mathbb{T}^2)} \right]^{\frac{1}{p}} dt \\
&\lesssim \int^T_0 \mathbb{E}^{\mathbb{P}_1} \left[  \| \Phi^{\Diamond k}_{N}(t) - \Phi^{\Diamond k}_{M}(t) \|^{2p}_{B^\alpha_{2p,2p}(\mathbb{T}^2)} \right]^{\frac{\gamma}{2p}} dt.  
\end{align*}
By Lemma \ref{lemm2}, we can calculate that
\begin{align}
& \frac{1}{k!} \mathbb{E}^{\mathbb{P}_1} \left[ \left( \left( \Phi_N^{\Diamond k} - \Phi_M^{\Diamond k} \right) (t,x) \right) \left( \left( \Phi_N^{\Diamond k} - \Phi_M^{\Diamond k} \right) (t,y) \right)  \right] \notag \\
&= \mathbb{E}^{\mathbb{P}_1} \left[ \Phi_{N} (t, x) \Phi_{N} (t, y) \right]^k - \mathbb{E}^{\mathbb{P}_1} \left[ \Phi_{M} (t, x) \Phi_{N} (t, y) \right]^k \notag \\
&\quad - \mathbb{E}^{\mathbb{P}_1} \left[ \Phi_{N} (t, x) \Phi_{M} (t, y) \right]^k + \mathbb{E}^{\mathbb{P}_1} \left[ \Phi_{M} (t, x) \Phi_{M} (t, y) \right]^k \notag \\
&= \mathbb{E}^{\mathbb{P}_1} \left[ \left( \Phi_N (t,x) - \Phi_M(t,x) \right)\Phi_N (t, y) \right] \sum^{k - 1}_{j = 0} \mathbb{E}^{\mathbb{P}_1} \left[ \Phi_{N} (t, x) \Phi_{N} (t, y) \right]^{k - j - 1} \mathbb{E}^{\mathbb{P}_1} \left[ \Phi_{M} (t, x) \Phi_N (t, y) \right]^j \notag \\
&\quad - \mathbb{E}^{\mathbb{P}_1} \left[  \left( \Phi_N (t,x) - \Phi_M(t,x) \right)\Phi_M (t, y) \right] \sum^{k - 1}_{j = 0} \mathbb{E}^{\mathbb{P}_1} \left[ \Phi_{N} (t, x) \Phi_{M} (t, y) \right]^{k - j - 1}  \mathbb{E}^{\mathbb{P}_1} \left[ \Phi_M (t, x) \Phi_M (t, y) \right]^j \notag 
\end{align}
By applying the Littlewood-Paley blocks $\Delta_{m,x}$ and $\Delta_{m,y}$ to the both sides and then letting $x = y$, we obtain
\begin{align*}
 &\frac{1}{k!} \mathbb{E}^{\mathbb{P}_1} \left[ \left| \Delta_m \left( \Phi^{\Diamond k}_N (t) - \Phi^{\Diamond k}_M (t) \right) \right|^2 \right] \\
&= \sum_{j = 0}^{k - 1}  \sum_{\substack{l_1, \cdots , l_k \in \mathbb{Z}^2 \\ l'_1 , \cdots , l'_k \in \mathbb{Z}^2}} \rho_m ( l_1 + \cdots + l_k ) \rho_m( l'_1 + \cdots + l'_k )e_{l_1 + \cdots +l_k + l'_1+ \cdots +l'_k} \\
&\quad \times J_1^{(N,M)} (l_1 , l'_1 ; t) \prod^{k - j}_{i = 2} J_2^{(N,N)} (l_i , l'_i ; t) \prod^{k}_{i = k - j + 1} J_2^{(M,N)} (l_i , l'_i ; t)  \\
&\quad + \sum_{j = 0}^{k - 1}  \sum_{\substack{l_1, \cdots , l_k \in \mathbb{Z}^2 \\ l'_1 , \cdots , l'_k \in \mathbb{Z}^2}} \rho_m \left( l_1 + \cdots + l_k \right) \rho_m \left( l'_1 + \cdots + l'_k  \right)e_{l_1 + \cdots + l_k + l'_1 + \cdots +l'_k} \\
&\quad \times J_1^{(M,N)} (l_1 , l'_1 ;t) \prod^{k - j}_{i = 2} J_2^{(N,M)} (l_i , l'_i ;t)\prod^{k}_{i = k - j + 1} J_2^{(M,M)} (l_i , l'_i ; t) 
\end{align*}
where we write
\[J_1^{(N,M)} (l,l';t) = \mathbb{E}^{\mathbb{P}_1} \left[ \left( \hat{\Phi}_N (l,t) - \hat{\Phi}_M (l,t) \right) \hat{\Phi}_M (l',t) \right], \]
\[J_2^{(N,M)} (l,l';t) = \mathbb{E}^{\mathbb{P}_1} \left[  \hat{\Phi}_N (l,t) \hat{\Phi}_M (l',t) \right], \]
and $\hat{\Phi}_N(l,t) \coloneqq \langle \Phi_N(t), e_l \rangle$ denotes the Fourier transformation of $\Phi_N(t)$.
By Lemma \ref{lemm1}, 
\begin{align*}
\mathbb{E}^{\mathbb{P}_1} \left[  \hat{\Phi} (l,t) \hat{\Phi} (l',t) \right] &= \mathbb{E}^{\mathbb{P}_1} \left[ \int^t_0 e^{(s-t)|l|^2} d\beta^l_L (s) \int^t_0 e^{(s-t)|l'|^2} d\beta^{l'}_L (s) \right] \\
&= \delta_{l, -l'} \int^t_0 e^{2(s-t)|l|^2} dL(s) \\
&= \delta_{l,-l'} \frac{1}{|l|^{2-2\epsilon}} \int^t_0 |l|^{2-2\epsilon}e^{2(s-t)|l|^2} dL(s) \\
&\lesssim \delta_{l,-l'} \frac{1}{|l|^{2-2\epsilon}} \int^t_0 \frac{1}{(t-s)^{1-\epsilon}} dL(s)
\end{align*}
for any $\epsilon>0$, where we use the fact that $\sup_{x\in \mathbb{R}_+} x^{2-2\epsilon} e^{-2x^2} < +\infty$.
Therefore, combining with the equality above, we obtain
\begin{align*}
& \mathbb{E}^{\mathbb{P}_1} \left[ \left| \Delta_m \left( \Phi^{\Diamond k}_N (t) - \Phi^{\Diamond k}_M (t) \right) \right|^2 \right] \\
&\lesssim \sum_{\substack{N<|l_1|\le M \\ l_2, \cdots , l_k \in \mathbb{Z}^2}} \rho_m ( l_1 + \cdots + l_k )^2 \frac{1}{|l_1|^{2-2\epsilon} \cdots |l_k|^{2-2\epsilon}} \left( \int^t_0 \frac{1}{(t-s)^{1-\epsilon}} dL(s) \right)^k .
\end{align*}
Thus, by the hypercontractivity of Gaussian polynomials,
\begin{align*}
&\int^T_0 \mathbb{E}^{\mathbb{P}_1} \left[  \| \Phi^{\Diamond k}_{N}(t) - \Phi^{\Diamond k}_{M}(t) \|^{2p}_{B^\alpha_{2p,2p}(\mathbb{T}^2)} \right]^{\frac{\gamma}{2p}} dt \\
&= \int^T_0 \mathbb{E}^{\mathbb{P}_1} \left[ \sum_{m\ge-1} 2^{2p\alpha m} \| \Delta_m\left(\Phi^{\Diamond k}_{N}(t) - \Phi^{\Diamond k}_{M}(t)\right) \|^{2p}_{L^{2p}(\mathbb{T}^2)} \right]^{\frac{\gamma}{2p}} dt \\
&= \int^T_0 \left( \int_{\mathbb{T}^2}  \sum_{m\ge-1} 2^{2p\alpha m}  \mathbb{E}^{\mathbb{P}_1} \left[ \left| \Delta_m\left(\Phi^{\Diamond k}_{N}(t) - \Phi^{\Diamond k}_{M}(t)\right) \right|^{2p} \right] dx \right)^{\frac{\gamma}{2p}} dt \\
&\lesssim \int^T_0 \left( \int_{\mathbb{T}^2}  \sum_{m\ge-1} 2^{2p\alpha m}  \mathbb{E}^{\mathbb{P}_1} \left[ \left| \Delta_m\left(\Phi^{\Diamond k}_{N}(t) - \Phi^{\Diamond k}_{M}(t)\right) \right|^2 \right]^p dx \right)^{\frac{\gamma}{2p}} dt \\
&\lesssim \int^T_0 \left(  \sum_{m\ge-1} 2^{2p\alpha m}  \left(\sum_{\substack{N<|l_1|\le M \\ l_2, \cdots , l_k \in \mathbb{Z}^2}} \rho_m ( l_1 + \cdots + l_k )^2 \frac{1}{|l_1|^{2-2\epsilon} \cdots |l_k|^{2-2\epsilon}}\right)^p   \left( \int^t_0 \frac{1}{(t-s)^{1-\epsilon}} dL(s) \right)^{pk} \right)^{\frac{\gamma}{2p}} dt \\
&= \int^T_0 \left( \int^t_0 \frac{1}{(t-s)^{1-\epsilon}} dL(s) \right)^{\frac{k\gamma}{2}} dt \left( \sum_{m\ge-1} 2^{2p\alpha m}\left(\sum_{\substack{N<|l_1|\le M \\ l_2, \cdots , l_k \in \mathbb{Z}^2}} \rho_m ( l_1 + \cdots + l_k )^2 \frac{1}{|l_1|^{2-2\epsilon} \cdots |l_k|^{2-2\epsilon}}\right)^p \right)^{\frac{\gamma}{2p}}. 
\end{align*}
Therefore, by the Besov embedding (Lemma \ref{emb}), we conclude that $\Phi_N^{\Diamond k}$ converges to some $\Phi^{\Diamond k}$ as $N\rightarrow \infty$ in $L^p\left( \Omega_1; L^\gamma ([0,T];B^\alpha_{\infty, \infty}(\mathbb{T}^2))\right)$ for fixed $\omega_2 \in \Omega_2$ if $\alpha<-\epsilon k, \gamma < \frac{2}{(1-\epsilon)k}$ for some $0<\epsilon<\frac{1}{k}$. Here, we use the fact that 
\begin{align}
\int^T_0 \left( \int^t_0 \frac{1}{(t-s)^{1-\epsilon}} dL(s) \right)^{\frac{k\gamma}{2}} dt \lesssim \int^T_0 \int^T_s \frac{1}{(t-s)^{\frac{k\gamma}{2(1-\epsilon)}}}dt dL(s) \lesssim L(T) < +\infty \label{yuyu}
\end{align} 
for $\gamma < \frac{2}{(1-\epsilon)k}$ and the fact that
the infinite sum
\begin{align}  \sum_{m\ge-1} 2^{2p\alpha m}\left(\sum_{l_1, l_2, \cdots , l_k \in \mathbb{Z}^2} \rho_m ( l_1 + \cdots + l_k )^2 \frac{1}{|l_1|^{2-2\epsilon} \cdots |l_k|^{2-2\epsilon}}\right)^p \label{haku} 
\end{align}
converges if $\alpha<-\epsilon k, 0<\epsilon<\frac{1}{k}$ where we used H\"older's inequality in \eqref{yuyu}, see also Remarks \ref{tuika} and \ref{tuika2} below.

Next, we prove the $\mathbb{P}_1$-almost-sure convergence for fixed $\omega_2$. Once we prove this, the $\mathbb{P}$-almost-sure convergence follows from Fubini's theorem. By the argument above, we obtain 
\[ \mathbb{E}^{\mathbb{P}_1} \left[ \left( \int^T_0  \| \Phi^{\Diamond k}_{N}(t) - \Phi^{\Diamond k}(t) \|^\gamma_{B^\alpha_{\infty,\infty}(\mathbb{T}^2)} dt\right)^p \right] \lesssim_{\omega_2} N^{-p\delta\gamma} \]
for some small $\delta>0$.
Therefore, for any $c>0$,
\begin{align*}
 \mathbb{P}_1 \left( \int^T_0  \| \Phi^{\Diamond k}_{N}(t) - \Phi^{\Diamond k}(t) \|^\gamma_{B^\alpha_{\infty,\infty}(\mathbb{T}^2)} dt > c \right) &\lesssim \mathbb{E}^{\mathbb{P}_1} \left[ \left( \int^T_0  \| \Phi^{\Diamond k}_{N}(t) - \Phi^{\Diamond k}(t) \|^\gamma_{B^\alpha_{\infty,\infty}(\mathbb{T}^2)} dt\right)^p \right] \\
&\lesssim_{\omega_2} N^{-p\delta \gamma} 
\end{align*}
by Chebyshev's inequality.
When we take sufficiently large $p$, the $\mathbb{P}_1$-almost-sure convergence follows from the Borel-Cantelli lemma.
\end{proof}

\begin{rem}\label{tuika}
The condition  $\gamma < \frac{2}{(1-\epsilon)k}$ in Proposition \ref{rere} is sharp in view of \eqref{yuyu} in the following sense: For example, if 
\begin{gather}
L(s) = 
\begin{cases}
0 \ \ \ \mathrm{when} \ 0\leq s < \frac{T}{2} \\
1 \ \ \ \mathrm{when} \ s\geq \frac{T}{2},
\end{cases}
\end{gather}
then
\begin{gather}
\int^t_0 \frac{1}{(t-s)^{1-\epsilon}}dL(s) = 
\begin{cases}
0 \ \ \ \mathrm{when} \ 0\leq t < \frac{T}{2} \\
\infty \ \ \ \mathrm{when} \ t = \frac{T}{2}\\
\frac{1}{(t-\frac{T}{2})^{1-\epsilon}} \ \ \ \mathrm{when} \ t > \frac{T}{2}.
\end{cases}
\end{gather}
Therefore, in this case
\[ \int^T_0 \left( \int^t_0 \frac{1}{(t-s)^{1-\epsilon}} dL(s)  \right)^{\frac{k\gamma}{2}}dt \geq \int^{T}_{\frac{T}{2}} \frac{1}{(t-\frac{T}{2})^{\frac{k\gamma(1-\epsilon)}{2}}}dt = \infty \]
if $\frac{k\gamma(1-\epsilon)}{2}\geq 1$. 
\end{rem}

\begin{rem}\label{tuika2}
The conditions $\alpha<-\epsilon k$ and $0<\epsilon<\frac{1}{k}$ in Proposition \ref{rere} come from \eqref{haku}: For any $\delta>0$, there holds
\[ |l_1 + \cdots + l_k|^{2(\alpha + \delta)}\rho_m(l_1 + \cdots + l_k)^2 \simeq 2^{2m(\alpha + \delta)} \rho_m(l_1 + \cdots + l_k)^2 . \]
Therefore,
\begin{align}
&\sum_{m\ge-1} 2^{2p\alpha m}\left(\sum_{l_1, l_2, \cdots , l_k \in \mathbb{Z}^2} \rho_m ( l_1 + \cdots + l_k )^2 \frac{1}{|l_1|^{2-2\epsilon} \cdots |l_k|^{2-2\epsilon}}\right)^p \notag \\ 
&\lesssim \sum_{m\ge-1} 2^{-2p\delta m}\left(\sum_{l_1, l_2, \cdots , l_k \in \mathbb{Z}^2} | l_1 + \cdots + l_k |^{2(\alpha + \delta)} \frac{1}{|l_1|^{2-2\epsilon} \cdots |l_k|^{2-2\epsilon}}\right)^p.\label{mkt}
\end{align}
On the other hand, because for any $\epsilon '>0$
\[ 1 = \sum_{m\ge -1} \rho_m(x) \leq \sum_{m\ge -1}\frac{2^{2m\delta}}{\epsilon '} \rho_m(x)^2 + \epsilon '\sum_{m\ge -1}2^{-2m\delta},\]
by choosing suffiecintly small $\epsilon '>0$, we obtain $1 \lesssim_\delta \sum_{m\ge -1}2^{2m\delta}\rho_m(x)^2$ uniformly in $x\in \mathbb{R}^d$. Thus, there holds
\begin{align}
&\sum_{l_1, l_2, \cdots , l_k \in \mathbb{Z}^2} | l_1 + \cdots + l_k |^{2(\alpha - 2\delta)} \frac{1}{|l_1|^{2-2\epsilon} \cdots |l_k|^{2-2\epsilon}}\notag \\ 
&\lesssim \sum_{m\ge -1}2^{2m\delta} \sum_{l_1, l_2, \cdots , l_k \in \mathbb{Z}^2} \rho_m(l_1 + \cdots + l_k)^2 | l_1 + \cdots + l_k |^{2(\alpha - 2\delta)} \frac{1}{|l_1|^{2-2\epsilon} \cdots |l_k|^{2-2\epsilon}}\notag \\ 
&\simeq \sum_{m\ge -1}2^{2m\delta + 2m(\alpha - 2\delta)}\sum_{l_1, l_2, \cdots , l_k \in \mathbb{Z}^2} \rho_m(l_1 + \cdots + l_k)^2 \frac{1}{|l_1|^{2-2\epsilon} \cdots |l_k|^{2-2\epsilon}}\notag \\
&\lesssim \left\{ \sum_{m\ge-1} 2^{2p\alpha m}\left(\sum_{l_1, l_2, \cdots , l_k \in \mathbb{Z}^2} \rho_m ( l_1 + \cdots + l_k )^2 \frac{1}{|l_1|^{2-2\epsilon} \cdots |l_k|^{2-2\epsilon}}\right)^p\right\}^{\frac{1}{p}}.\label{mkt2}
\end{align}
After all, we obtain
\begin{align*}
&\left(\sum_{l_1, l_2, \cdots , l_k \in \mathbb{Z}^2} | l_1 + \cdots + l_k |^{2(\alpha - 2\delta)} \frac{1}{|l_1|^{2-2\epsilon} \cdots |l_k|^{2-2\epsilon}}\right)^p \\
&\lesssim \sum_{m\ge-1} 2^{2p\alpha m}\left(\sum_{l_1, l_2, \cdots , l_k \in \mathbb{Z}^2} \rho_m ( l_1 + \cdots + l_k )^2 \frac{1}{|l_1|^{2-2\epsilon} \cdots |l_k|^{2-2\epsilon}}\right)^p \\
&\lesssim \left(\sum_{l_1, l_2, \cdots , l_k \in \mathbb{Z}^2} | l_1 + \cdots + l_k |^{2(\alpha + \delta)} \frac{1}{|l_1|^{2-2\epsilon} \cdots |l_k|^{2-2\epsilon}}\right)^p
\end{align*}
for any $\delta>0$ from \eqref{mkt} and \eqref{mkt2}. And one can check that the infinite sum
\[ \sum_{l_1, l_2, \cdots , l_k \in \mathbb{Z}^2} | l_1 + \cdots + l_k |^{2\beta} \frac{1}{|l_1|^{2-2\epsilon} \cdots |l_k|^{2-2\epsilon}} \]
converges if $\beta< -\epsilon k$ and $\epsilon< \frac{1}{k}$ by applying \cite[Lemma 2.3]{3dwave}, for example.
\end{rem}

As a corollary of Proposition \ref{rere}, we can see that $\Phi$ is indeed a distribution, not a function in the following sense.
\begin{cor}\label{coco}
For any $T>0$ and $\omega_2 \in \Omega_2$ with $L(T-)(\omega_2) \coloneqq \lim_{t \nearrow T}L(t)(\omega_2) >0$, there holds $\Phi \notin L^2([0,T]\times \mathbb{T}^2)$ for $\mathbb{P}_1$-almost surely.
\end{cor}
\begin{proof}
Fix $\omega_2 \in \Omega_2$ and $\delta>0$ with $L(T-\delta)(\omega_2)>0$.
By Proposition \ref{rere}, $\Phi_N^{\Diamond 2} = \Phi^2_N - c^H_N$ converges in $L^1([0,T];B^{-\epsilon}_{\infty,\infty})$ for $\mathbb{P}_1$-almost surely.
On the other hand, there holds
\begin{align*}
\int^T_0 c^H_N (t) dt &= \frac{1}{2\pi} \sum_{|l| \le N} \int^T_0 \int^t_0 e^{2(s-t)|l|^2} dL(s) dt \\
&= \frac{1}{2\pi} \sum_{|l| \le N} \int^T_0 \int^T_s e^{2(s-t)|l|^2} dt dL (s) \\
&= \frac{1}{2\pi} \sum_{|l| \le N} \int^T_0 \frac{1}{2|l|^2} \left( 1 - e^{2(s-T)|l|^2} \right) dL(s) \\
&\ge \frac{1}{4\pi} \sum_{|l|\le N} \frac{1}{|l|^2} \int^{T-\delta}_0 \left( 1 - e^{-2\delta |l|^2} \right) dL(s) \ \nearrow + \infty \ \ \mathrm{as}\ N \rightarrow \infty .
\end{align*}
Therefore, we conclude that $\Phi \notin L^2([0,T]\times\mathbb{T}^2)$ for $\mathbb{P}_1$-almost surely.
\end{proof}

In the rest of this section, we prove the convergence of $\Psi_N^{\Diamond k}$ in \eqref{qq}.

\begin{prop}\label{yy}
Fix $T > 0$ and $k \in \mathbb{Z}_{>0}$.
For any $\alpha<0$, $p\ge 1$, and fixed $\omega_2 \in \Omega_2$, $\Psi_N^{\Diamond k} (\omega_1, \omega_2)$ converges in $L^p \left( \Omega_1; C ([0,T];B^\alpha_{\infty, \infty}(\mathbb{T}^2)) \right)$ and $\mathbb{P}_1$-almost surely.
In particular, $\Psi_N^{ \Diamond k}$ converges in $C ([0,T];B^\alpha_{\infty, \infty}(\mathbb{T}^2))$ $\mathbb{P}$-almost surely.
\end{prop}
\begin{rem}
Differently from the case of the heat equation, $\Psi$ has time-continuity, whether or not the subordinator $L$ has continuous paths. See Section 6.2 for the explanation of the discontinuity of the solution of the linear stochastic heat equation.
\end{rem}

\begin{rem}\label{hv}
In \cite{wave}, they show the convergence of the renormalized powers of $\Psi_N$ on the Sobolev space $W^{\alpha, \infty}(\mathbb{T}^2)$ instead of the Besov space $B^{\alpha}_{\infty,\infty}(\mathbb{T}^2)$. By the same kind of argument as in the proof below, one can also prove the convergence in $C ([0,T];W^{\alpha, \infty}(\mathbb{T}^2))$.
\end{rem}
\begin{proof}
Fix $\omega_2 \in \Omega_2$ and $1\le p < \infty$. To apply the Kolmogorov's continuity theorem, we consider the difference
\[ d_{s,t} \Psi_N^{\Diamond k} \coloneqq \Psi_N^{\Diamond k}(t) - \Psi_N^{\Diamond k}(s) . \]
For the simplicity of  notation, we derive the estimate on $\Psi_N^{\Diamond k}$, instead of the difference $\Psi_N^{\Diamond k} - \Psi_M^{\Diamond k}$.
By the hypercontractivity of Gaussian polynomials, there holds
\begin{align}
\mathbb{E}^{\mathbb{P}_1} \left[ \| d_{s,t} \Psi_N^{\Diamond k} \|_{B^{\alpha}_{2p,2p}(\mathbb{T}^2)}^{2p} \right] &= \mathbb{E}^{\mathbb{P}_1} \left[ \sum_{m\ge -1} 2^{2p\alpha m} \left\| \Delta_m \left( d_{s,t} \Psi_N^{\Diamond k}\right) \right\|_{L^{2p}(\mathbb{T}^2)}^{2p} \right] \notag \\
&\lesssim \sum_{m\ge -1} 2^{2p\alpha m} \int_{\mathbb{T}^2} \mathbb{E}^{\mathbb{P}_1} \left[ \left| \Delta_m \left( d_{s,t} \Psi_N^{\Diamond k}\right) (x)\right|^2 \right]^p dx . \label{kaka}
\end{align}
Then, by a similar argument to the proof of Theorem \ref{rere}, it can be shown that
\begin{align*}
&\frac{1}{k!}\mathbb{E}^{\mathbb{P}_1} \left[ \left| \Delta_m \left( d_{s,t} \Psi_N^{\Diamond k}\right) \right|^2 \right] \\
&=  \sum_{j = 0}^{k - 1}  \sum_{\substack{l_1, \cdots , l_k \in \mathbb{Z}^2 \\ l'_1 , \cdots , l'_k \in \mathbb{Z}^2}} \rho_m ( l_1 + \cdots + l_k ) \rho_m( l'_1 + \cdots + l'_k )e_{l_1 + \cdots + l_k + l'_1 + \cdots + l'_k} \\
&\quad \times G_1 (l_1 , l'_1 ; s,t) \prod^{k - j}_{i = 2} G_2 (l_i , l'_i ; t,t) \prod^{k}_{i = k - j + 1} G_2 (l_i , l'_i ; s,t)  \\
&\quad + \sum_{j = 0}^{k - 1}  \sum_{\substack{l_1, \cdots , l_k \in \mathbb{Z}^2 \\ l'_1 , \cdots , l'_k \in \mathbb{Z}^2}} \rho_m \left( l_1 + \cdots + l_k \right) \rho_m \left( l'_1 + \cdots + l'_k  \right)e_{l_1 + \cdots + l_k + l'_1 + \cdots + l'_k} \\
&\quad \times G_1 (l_1 , l'_1 ;t,s) \prod^{k - j}_{i = 2} G_2 (l_i , l'_i ;t,s)\prod^{k}_{i = k - j + 1} G_2 (l_i , l'_i ; s,s) 
\end{align*}
where we write
\[G_1 (l,l';s,t) = \mathbb{E}^{\mathbb{P}_1} \left[ d_{s,t} \hat{\Psi}_N (l) \hat{\Psi}_N (l',t) \right] \] and
\[G_2 (l,l';s,t) = \mathbb{E}^{\mathbb{P}_1} \left[  \hat{\Psi}_N (l,s) \hat{\Psi}_N (l',t) \right]. \]

By Lemma \ref{lemm1}, for any $u_1, u_2 \in \{s,t\}$, 
\begin{align}
\left| \mathbb{E}^{\mathbb{P}_1} \left[ \hat{\Psi} (l, u_1) \hat{\Psi} (l', u_2)  \right] \right| &= \left| \mathbb{E}^{\mathbb{P}_1} \left[ \int^{u_1}_0 \frac{\sin((u_1 - u ) |l| )}{|l|} d\beta^l_L (u)  \int^{u_2}_0 \frac{\sin((u_2 - u ) |l'| )}{|l|} d\beta^{l'}_L (u) \right] \right| \notag \\
&\le \delta_{l, -l'} \int^{u_1 \wedge u_2 }_0 \frac{| \sin((u_1 - u ) |l| ) \sin((u_2 - u ) |l| ) }{|l|^2 } dL(u) \lesssim \delta_{l, -l'}\frac{1}{|l|^2} L(T) \notag
\end{align}
and by the triangle inequality,
\begin{align}
\left| \mathbb{E}^{\mathbb{P}_1} \left[ d_{s, t} \hat{\Psi} (l) \hat{\Psi} (l', u_1) \right] \right| &= \left| \mathbb{E}^{\mathbb{P}_1} \left[ \left( \int^{t}_0 \frac{\sin((t - u) |l| )}{|l|} d\beta^l_L (u) - \int^{s}_0 \frac{\sin((s - u) |l| )}{|l|} d\beta^l_L (u) \right) \right. \right.\notag \\
&\quad \quad \quad \quad \quad \quad \quad \quad \quad \quad \quad \quad \quad \quad \quad \quad \quad \quad \quad \times \left. \left. \int^{u_1}_0 \frac{\sin((u_1 - u) |l'| )}{|l'|} d\beta^{l'}_L (u)  \right] \right| \notag \\
&\le \left| \mathbb{E}^{\mathbb{P}_1} \left[ \int^{t}_{s } \frac{\sin((t - u) |l| )}{|l|} d\beta^l_L (u) \int^{u_1}_0 \frac{\sin((u_1 - u) |l'| )}{|l'|} d\beta^{l'}_L (u) \right] \right| \notag \\
&\  + \left| \mathbb{E}^{\mathbb{P}_1} \left[ \int^{s}_0 \frac{\sin((t - u) |l| ) - \sin((s - u) |l| )}{|l|} d\beta^l_L (u) \int^{u_1}_0 \frac{\sin((u_1 - u) |l'| )}{|l'|} d\beta^{l'}_L (u) \right] \right|. \notag 
\end{align}
For any fixed $\rho \in (0,1]$, the first term of the last line can be dominated as
\begin{align}
&\left| \mathbb{E}^{\mathbb{P}_1} \left[ \int^{t}_{s} \frac{\sin((t - u) |l| )}{|l|} d\beta^l_L (u) \int^{u_1}_0 \frac{\sin((u_1 - u) |l'| )}{|l'|} d\beta^{l'}_L (u) \right] \right| \notag \\
&\lesssim \delta_{l, -l'} \int^{u_1}_{s} \frac{| \sin((t - u) |l|) |}{|l|^2} dL(u) \notag \\
&\lesssim \delta_{l, -l'} \int^{u_1}_{s} \frac{|t - u|^\rho}{|l|^{2 - \rho}} dL(u) \notag \\
&\lesssim \delta_{l, -l'} \frac{|t - s|^\rho}{|l|^{2 - \rho}} L(T) \notag
\end{align}
where we use the fact $|\sin \theta | \lesssim |\theta |^\rho$. 
Similarly, the second term can be dominated as
\begin{align}
& \left| \mathbb{E}^{\mathbb{P}_1} \left[ \int^{s}_0 \frac{\sin((t - u) |l| ) - \sin((s - u) |l| )}{|l|} d\beta^l_L (u) \int^{u_1}_0 \frac{\sin((u_1 - u) |l'| )}{|l'|} d\beta^{l'}_L (u) \right] \right| \notag \\
&\lesssim \delta_{l, -l'} \int^{s}_0 \frac{| \sin((t - u) |l| ) - \sin((s - u) |l| )|}{|l|^2} dL(u) \notag \\
&\lesssim \delta_{l, -l'} \int^{s }_0 \frac{|t - s|^\rho}{|l|^{2 - \rho}} dL(u) \lesssim \delta_{l, -l'}\frac{|t - s|^\rho}{|l|^{2 - \rho}} L(T). \notag
\end{align}
Therefore,
\[ \mathbb{E}^{\mathbb{P}_1} \left[ \left| \Delta_m \left( d_{s,t} \Psi_N^{\Diamond k}\right) \right|^2 \right] \lesssim |t - s|^\rho \sum_{|l_1|, \cdots , |l_k| \le N  } \rho_m ( l_1 + \cdots + l_k )^2 \frac{1}{| l_1 |^{2 - \rho} | l_2 |^2 \cdots |l_k |^2 } L(T)^k .\]
By combining with the inequality \eqref{kaka}, we obtain
\begin{align*}
&\mathbb{E}^{\mathbb{P}_1} \left[ \| d_{s,t} \Psi_N^{\Diamond k} \|_{B^{\alpha}_{2p,2p}(\mathbb{T}^2)}^{2p} \right] \\
&\lesssim \sum_{m\ge -1} 2^{2p\alpha m}\left( |t - s|^\rho \sum_{|l_1|, \cdots , |l_k| \le N } \rho_m ( l_1 + \cdots + l_k )^2 \frac{1}{| l_1 |^{2 - \rho} | l_2 |^2 \cdots |l_k |^2 } L(T)^k \right)^p \\
&= |t - s|^{\rho p} L(T)^{kp} \sum_{m\ge -1} 2^{2p\alpha m}\left(\sum_{|l_1|, \cdots , |l_k| \le N } \rho_m ( l_1 + \cdots + l_k )^2 \frac{1}{| l_1 |^{2 - \rho} | l_2 |^2 \cdots |l_k |^2 } \right)^p
\end{align*}
uniformly in $N$. By a similar argument, we obtain
\begin{align*}
&\mathbb{E}^{\mathbb{P}_1} \left[ \| d_{s,t} \left( \Psi_N^{\Diamond k} - \Psi_M^{\Diamond k}\right) \|_{B^{\alpha}_{2p,2p}(\mathbb{T}^2)}^{2p} \right] \\
&\lesssim |t - s|^{\rho p} L(T)^{kp} \sum_{m\ge -1} 2^{2p\alpha m}\left(\sum_{N<|l_1| \le M } \rho_m ( l_1 + \cdots + l_k )^2 \frac{1}{| l_1 |^{2 - \rho} | l_2 |^2 \cdots |l_k |^2 } \right)^p
\end{align*}
for any $M>N\ge 1$.
Therefore, by taking sufficiently large $p$, convergence in $L^p(\Omega_1)$ follows from Kolmogorov's theorem and the Besov embedding. The almost-sure convergence can also be shown by a similar argument to the proof of Proposition \ref{rere}.
\end{proof}
\begin{cor}
For any $T>0$ and $\omega_2 \in \Omega_2$ with $L(T-)(\omega_2) \coloneqq \lim_{t \nearrow T}L(t)(\omega_2) >0$, there holds $\Psi \notin L^2([0,T]\times \mathbb{T}^2)$ for $\mathbb{P}_1$-almost surely.
\end{cor}
\begin{proof}
The proof is similar to that of Corollary \ref{coco}.
\end{proof}

\subsection{Stationary solutions of the linear equations}
In Sections 4.1 and 4.2, we dealt with the solutions of \eqref{rehe} and \eqref{rehe2} under the initial conditions $0$. However, by a slight modification of the equations, we can consider the stationary solutions under some assumption on the subordinator $L$. In the context of the usual Wick renormalization i.e. when $L(t)=t$, the advantage of considering the stationary solution of the linear equation is that the renormalization constant can be taken independently of $t$. In our situation, although it is hopeless to take renormalization constants $c_N^H$ and $c^W_N$ independently of $t$ due to the randomness of $L$, we can take $c_N^H$ and $c^W_N$ as $\mathbb{R}_+$-valued stationary processes.  

To consider the stationary solutions, we have to replace the massless Laplacian $\Delta$ with the massive one in \eqref{rehe} i.e. we consider the equation
\begin{gather}
\begin{cases}\label{dadan}
\partial_t  \Phi_N + (1-\Delta) \Phi_N =  P_N\partial_t{W_L}\\
\Phi_N(0)=\phi_0.
\end{cases}
\end{gather}
For the wave equation \eqref{rehe2}, replacing the Laplacian is not enough and we have to add the ``damping term'' $\partial_t\Psi_N$ to the equation i.e. we consider the stochastic linear damed wave equation
\begin{gather}
\begin{cases} \label{dadan2}
\partial_t ^2 \Psi_N + \partial_t\Psi_N + (1- \Delta) \Psi_N = P_N\partial_t{W_L} \\
(\Psi_N (0), \partial_t \Psi_N (0)) = (\psi_0, \psi_1). 
\end{cases}
\end{gather}
Moreover, in view of Theorem \ref{ou}, we assume that $L$ is a L\'evy process with nondecreasing paths satisfying
\begin{equation}\label{con}
\int_{[0.\infty )} \left(0\lor\log|x| \right) \rho(dx)<+\infty
\end{equation}
where $\rho$ is a L\'evy measure associated to $L$. More precisely, $\rho$ is determined by the L\'evy-Khintchine representation of $L$
\begin{equation}\notag
\mathbb{E}[ e^{\sqrt{-1} z L(1)}] =
\exp \left[ \left\{ \sqrt{-1} az  -\frac{1}{2} bz^2 + \int_{\mathbb{R}} \eta (z,x) \rho(dx)\right\}\right] 
\end{equation}
where $a \in \mathbb{R}, b \in \mathbb{R}_+$ and
\[ \eta (z,x) = e^{\sqrt{-1} z x} - 1 - \sqrt{-1}\frac{ zx }{1+|x|^2}. \]
Note that $W_L$ is a $\mathcal{D}'(\mathbb{T}^2)$-valued L\'evy process in this case as explained in Section 3.1.
In the following, we extend the time parameter $t$ of $W$ and $L$ to the whole line $\mathbb{R}$: Let $\{\tilde{W}_t\}_{t\in\mathbb{R}_+}$ be a cylindrical Brownian motion defined on $\Omega_1$ which is independent of $W$ and $\{\tilde{L_t}\}_{t\in\mathbb{R}_+}$ be a  L\'evy process with $L(1)\sim\tilde{L}(1)$ defined on $\Omega_2$ which is independent of $L$. Then, we extend $W$ and $L$ by
\[ W(t) = \tilde{W}(-t), \ L(t) = \tilde{L}(-t)\]
for $t<0$.
\begin{prop}\label{kj}
We assume \eqref{con}. Then, 
\begin{align}
\Phi_N^{stat.}(t) \coloneqq \int^t_{-\infty} e^{(t-s)(\Delta - 1)} P_N dW_L(s) = \frac{1}{2\pi}\sum_{|l|\le N} \int^t_{-\infty} e^{(s-t)(|l|^2+1)}d\beta^l_L(s)e_l
\end{align}
is a stationary solution of \eqref{dadan}.
Moreover, there holds
\[ \mathbb{E}^{\mathbb{P}_1}\left[ |\Phi_N^{stat.}(t,x)|^2 \right] = \frac{1}{(2\pi)^2}\sum_{|l|\le N} \int^t_{-\infty} e^{2(s-t)(|l|^2+1)}dL(s) < +\infty \ \ \ \mathbb{P}_2\mbox{-a.s.}. \] 
\end{prop}
\begin{proof}
We can see the SPDE \eqref{dadan} as the Ornstein-Uhlenbeck type SDE on the finite-dimensional subspace span$\{e_l\}_{|l|\le N}$.
By Theorem \ref{ou}, the first statement follows once we prove the existence of the integral
\begin{align}
\int^t_{-\infty} e^{(s-t)(|l|^2+1)} dB_L(s) \label{rcep}
\end{align}
for any $l$ and $\mathbb{R}$-valued standard Brownian motions $B$ defined on $\Omega_1$.
We fix $t$ and write
\[ A(u) \coloneqq \int^t_{u} e^{(s-t)(|l|^2+1)} dB_L(s) \]
for $u\in \mathbb{R}_-$.
By Lemma \ref{lemm1}, 
\begin{equation} \label{j}
\mathbb{E}^{\mathbb{P}_1} \left[ |A(u_1) - A(u_2)|^2 \right] = \int^{u_1}_{u_2} e^{2(s-t)(|l|^2+1)}dL(s) .
\end{equation}
By the assumption \eqref{con} and Theorem \ref{ou}, the right hand side of \eqref{j} converges to $0$ as $u_1,u_2 \rightarrow -\infty$.
Therefore, for any $\epsilon>0$,
\begin{align}
\mathbb{P}\left(|A(u_1)-A(u_2)|>\epsilon\right) \lesssim \mathbb{E}^{\mathbb{P}}\left[|A(u_1)-A(u_2)|^2 \wedge 1\right] = \mathbb{E}^{\mathbb{P}_2}\left[ \mathbb{E}^{\mathbb{P}_1}\left[|A(u_1)-A(u_2)|^2 \right]\wedge 1\right] \rightarrow 0
\end{align}
as $u_1,u_2 \rightarrow -\infty$, where we use the Markov inequality. Thus, the integral \eqref{rcep} exits.
The second statement follows from Lemma \ref{lemm1} and an easy computation.
\end{proof}
\begin{prop}\label{ash}
We assume \eqref{con}. Then, 
\begin{equation}
\Psi_N^{stat.}(t) \coloneqq \int^t_{-\infty} \mathcal{D}(t-s) P_NdW_L(s) = \frac{1}{2\pi}\sum_{|l|\le N} \int^t_{-\infty} e^{\frac{1}{2}(s-t)}\frac{\sin \left((t-s)\sqrt{\frac{3}{4}+|l|^2}\right)}{\sqrt{\frac{3}{4}+|l|^2}} d\beta^l_L(s)e_l
\end{equation}
is a stationary solution of \eqref{dadan2}, where we write
\[ \mathcal{D}(t) = e^{-\frac{1}{2}t}\frac{\sin \left(t\sqrt{\frac{3}{4}-\Delta}\right)}{\sqrt{\frac{3}{4}-\Delta}}.\]
Moreover, there holds
\[ \mathbb{E}^{\mathbb{P}_1}\left[ |\Psi_N^{stat.}(t,x)|^2 \right] = \frac{1}{(2\pi)^2}\sum_{|l|\le N} \int^t_{-\infty} e^{(s-t)}\frac{\sin^2 \left((t-s)\sqrt{\frac{3}{4}+|l|^2}\right)}{\frac{3}{4}+|l|^2}dL(s) < +\infty \ \ \ \mathbb{P}_2\mbox{-a.s.}. \] 
\end{prop}
\begin{proof}
By the Duhamel principle, the mild solution of the Cauchy problem \eqref{dadan2} is written by
\begin{equation}
\Psi_N (t) = \partial_t \mathcal{D}(t) \psi_0 + \mathcal{D}(t) (\psi_0 + \psi_1) + \int^t_0 \mathcal{D}(t-s) P_N dW_L(s)
\end{equation}
where we write 
\[ \partial_t \mathcal{D}(t) = e^{-\frac{1}{2}t} \left( \cos \left( t \sqrt{\frac{3}{4}-\Delta} \right) - \frac{\sin \left( t \sqrt{\frac{3}{4}-\Delta} \right)}{2\sqrt{\frac{3}{4}-\Delta}} \right).\]
Letting
\[ \psi_0^{stat.} = \int^0_{-\infty} \mathcal{D}(-s) P_NdW_L(s) \]
and
\[ \psi_1^{stat.} = \int^0_{-\infty} \partial_t \mathcal{D}(-s) P_NdW_L(s), \]
(The existence of the integrals follows from a similar argument to the proof of Proposition \ref{kj}.) we define
\begin{equation}
\Psi_N^{stat.} (t) \coloneqq \partial_t \mathcal{D}(t) \psi_0^{stat.} + \mathcal{D}(t) (\psi_0^{stat.} + \psi_1^{stat.}) + \int^t_0 \mathcal{D}(t-s) P_N dW_L(s).
\end{equation}
Then, it is easily checked that 
\begin{equation}
\Psi_N^{stat.} (t) = \int^t_{-\infty} \mathcal{D}(t-s) P_N dW_L(s)
\end{equation}
and $\Psi_N^{stat.}$ is a stationary solution of \eqref{dadan2}.
\end{proof}
It is easy to see that
\[ c^{H, stat.}_N(t) \coloneqq \frac{1}{(2\pi)^2}\sum_{|l|\le N} \int^t_{-\infty} e^{2(s-t)(|l|^2+1)}dL(s)\]
and
\[ c^{W, stat.}_N(t) \coloneqq \frac{1}{(2\pi)^2}\sum_{|l|\le N} \int^t_{-\infty} e^{(s-t)}\frac{\sin^2 \left((t-s)\sqrt{\frac{3}{4}+|l|^2}\right)}{\frac{3}{4}+|l|^2}dL(s)\]
are $\mathbb{R}_+$-valued stationary stochastic processes and we can show that
\begin{enumerate}
\item
For $\Phi^{stat.}_N$ in Proposition \textup{\ref{kj}} and $c^{H, stat.}_N$, the statement of Proposition \textup{\ref{rere}} holds.
\item
For $\Psi^{stat.}_N$ in Proposition \textup{\ref{ash}} and $c^{W, stat.}_N$, the statement of Proposition \textup{\ref{yy}} holds.
\end{enumerate}
The precise statement is as follows.

\begin{prop}
We assume \eqref{con} and fix $T > 0$ and $k \in \mathbb{Z}_{>0}$. We define the renormalized powers of $\Phi^{stat.}_N$ and $\Psi^{stat.}_N$ by 
\[ \Phi^{stat. \Diamond k}_N \coloneqq H_k (\Phi^{stat.}_N;c^{H, stat.}_N)\quad \mbox{and} \quad \Psi^{stat. \Diamond k}_N \coloneqq H_k (\Psi^{stat.}_N;c^{W, stat.}_N). \]
Then, the following statements hold:
\begin{enumerate}
\item
For any $0<\epsilon<\frac{1}{k},\ \alpha<-\epsilon k,\ 0<\gamma < \frac{2}{(1-\epsilon)k}$, $p\ge 1$, and fixed $\omega_2 \in \Omega_2$, $\Phi_N^{stat. \Diamond k} (\omega_1, \omega_2)$ converges in $L^p \left( \Omega_1; L^\gamma ([0,T];B^\alpha_{\infty, \infty}(\mathbb{T}^2)) \right)$ and $\mathbb{P}_1$-almost surely.
In particular, $\Phi_N^{stat. \Diamond k}$ converges in $L^\gamma ([0,T];B^\alpha_{\infty, \infty}(\mathbb{T}^2))$ $\mathbb{P}$-almost surely.
\item
For any $\alpha<0$, $p\ge 1$, and fixed $\omega_2 \in \Omega_2$, $\Psi_N^{stat. \Diamond k} (\omega_1, \omega_2)$ converges in $L^p \left( \Omega_1; C ([0,T];B^\alpha_{\infty, \infty}(\mathbb{T}^2)) \right)$ and $\mathbb{P}_1$-almost surely.
In particular, $\Psi_N^{stat. \Diamond k}$ converges in $C ([0,T];B^\alpha_{\infty, \infty}(\mathbb{T}^2))$ $\mathbb{P}$-almost surely.
\end{enumerate}
\end{prop}
The proof is essentially same as the proof of Propositions \ref{rere} and \ref{yy}.

\section{Construction of the local solutions}
By applying Propositions \ref{rere} and \ref{yy}, we can prove the local well-posedness of the renormalized equation.

\subsection{Local well-posedness of the stochastic nonlinear heat equations}
In this subsection, we consider the following renormalized equation:
\begin{align} \label{la}
v(t) = e^{t\Delta} u_0 \pm \sum^k_{j=0} \binom{k}{l}\int^t_0 e^{(t-s)\Delta} v^j (s) \Phi^{\Diamond (k-j)} (s) ds
\end{align}
Note that it is the mild formulation of the random PDE 
\begin{gather}
\begin{cases} 
(\partial_t - \Delta )v =  \sum^k_{j=0} \binom{k}{l} v^j \Phi^{\Diamond (k-j)} \\
v(0) = u_0 \ .
\end{cases} 
\end{gather}
For $k\ge 3$, the term $\int^t_0 e^{(t-s)\Delta} \Phi^{\Diamond k} (s) ds$ in \eqref{la} does not make sense due to the lack of time-integrability of $\Phi^{\Diamond k}$. Indeed, we can only show $\Phi^{\Diamond k} \in L^\gamma ([0,T];B^\alpha_{\infty,\infty}(\mathbb{T}^2))$ for $\gamma<1$ in view of Proposition \ref{rere}. 
Because of this situation, we only consider the case $k=2$:
\begin{align} \label{la2}
v(t) = e^{t\Delta} u_0 + \int^t_0 e^{(t-s)\Delta} v^2 (s)  ds +  2\int^t_0 e^{(t-s)\Delta} v (s) \Phi (s) ds + \int^t_0 e^{(t-s)\Delta} \Phi^{\Diamond 2} (s) ds
\end{align}
By Proposition \ref{rere}, it is sufficient to consider the deterministic equation
\begin{equation}\label{aaa}
v(t) = e^{t\Delta} u_0 + \int^t_0 e^{(t-s)\Delta} v^2 (s)  ds +  2\int^t_0 e^{(t-s)\Delta} v (s) \Xi_1 (s) ds + \int^t_0 e^{(t-s)\Delta} \Xi_2 (s) ds
\end{equation}
for a given data $(\Xi_1, \Xi_2)\in L^{\frac{2}{1-\epsilon}}([0,T];B^{-\epsilon}_{\infty,\infty}(\mathbb{T}^2)) \times L^{\frac{1}{1-\epsilon}}([0,T];B^{-2\epsilon}_{\infty,\infty}(\mathbb{T}^2))$ with $0<\epsilon<\frac{1}{2}$.

To solve \eqref{aaa}, we derive the following type of Schauder estimate.
\begin{lemm}\label{al}
Let
\[ u(t) = e^{t\Delta}u_0 + \int^t_0 e^{(t-s)\Delta}f(s) ds .\]
Then, for any $\epsilon>0, \theta \in \mathbb{R}, r_1\in(1,\infty]$, and $r_2,r_3\in[1,\infty)$ with $1+\frac{1}{r_2} = \frac{1}{r_1} + \frac{1}{r_3}$, there holds
\[ \| u \|_{L^{r_2} (\left[0,T\right];B^{\theta+2/r_3-\epsilon}_{p,q} ) \cap C([0,T];B^{\theta+2(1-1/r_1)-\epsilon}_{p,q})} \lesssim \|u_0 \|_{B^{\theta+2/r_3 - \epsilon}_{p,q}} + \| f \|_{L^{r_1}([0,T];B^\theta_{p,q})} .\]
\end{lemm}
\begin{proof}
Let $p,q\in [1,\infty]$. We use the following estimate: 
For any $\delta \ge 0$ and $s\in\mathbb{R}$,
\[\|e^{t\Delta}u\|_{B^{s+2\delta}_{p,q}} \lesssim t^{-\delta}\|u\|_{B^s_{p,q}}\] uniformly in $t\in[0,T]$.
For the proof of this, see  \cite[Proposition 5]{weber}.
By this estimate, 
\[ \|e^{t\Delta}u_0\|_{B^{\theta+2/r_3-\epsilon}_{p,q}} \lesssim \|u_0\|_{B^{\theta+2/r_3-\epsilon}_{p,q}}\] and
\[ \left\|\int^t_0 e^{(t-s)\Delta}f(s) ds\right\|_{B^{\theta+2/r_3-\epsilon}_{p,q}} \lesssim \int^t_0 (t-s)^{-\left(\frac{1}{r_3}-\frac{\epsilon}{2}\right)} \|f(s)\|_{B^\theta_{p,q}} ds .\]
Therefore, by Young's inequality, 
\begin{align*}
\left\| \int^t_0 e^{(t-s)\Delta}f(s) ds \right\|_{L^{r_2} (\left[0,T\right];B^{\theta+2/r_3-\epsilon}_{p,q} )} &\lesssim \| t \mapsto t^{-\left(\frac{1}{r_3}-\frac{\epsilon}{2}\right)}\|_{L^{r_3}[0,T]}\|f\|_{L^{r_1}([0,T];B^\theta_{p,q})} \\
&\lesssim \|f\|_{L^{r_1}([0,T];B^\theta_{p,q})}.
\end{align*}
For the estimate on $C([0,T];B^{\theta+2(1-1/r_1)-\epsilon}_{p,q})$, see \cite[Proposition A.3]{exp}.
\end{proof}
Now, we can solve the equation \eqref{la2}.
\begin{prop}\label{fo}
For given $\Xi_1 \in L^{\frac{2}{1-\epsilon}}([0,T];B^{-\epsilon}_{\infty,\infty}(\mathbb{T}^2)), \Xi_2 \in L^{\frac{1}{1-\epsilon}}([0,T];B^{-2\epsilon}_{\infty,\infty}(\mathbb{T}^2))$ with $ 0<\epsilon<\frac{1}{2}$, there exists some small $T>0$ such that the equation \eqref{la2} has a unique solution in
\[L^\gamma (\left[0,T\right];B^{2/\gamma-\delta}_{\infty,\infty}(\mathbb{T}^2) ) \cap C([0,T];B^{-\delta}_{\infty,\infty}(\mathbb{T}^2))\]
for any $\frac{2}{1-\epsilon}<\gamma<\frac{2}{\epsilon}$, $0<\delta<\frac{2}{\gamma}-\epsilon$ and initial conditions $u_0 \in B^{2/\gamma-\delta}_{\infty,\infty}(\mathbb{T}^2)$.
\end{prop}
\begin{proof}
We define the map
\[ v \mapsto \Gamma(v) \coloneqq e^{t\Delta} u_0 + \int^t_0 e^{(t-s)\Delta} v^2 (s)  ds +  2\int^t_0 e^{(t-s)\Delta} v (s) \Xi_1 (s) ds + \int^t_0 e^{(t-s)\Delta} \Xi_2 (s) ds  \]
and let $X^{\gamma,\delta}(T) \coloneqq L^\gamma (\left[0,T\right];B^{2/\gamma-\delta}_{\infty,\infty}(\mathbb{T}^2) ) \cap C([0,T];B^{-\delta}_{\infty,\infty}(\mathbb{T}^2))$.
By Lemmas \ref{al} and \ref{pro}, there holds
\begin{align*}
\|e^{t\Delta}u_0\|_{X^{\gamma,\delta}(T)} \lesssim \|u_0\|_{B^{2/\gamma-\delta}_{\infty,\infty}(\mathbb{T}^2)},
\end{align*}

\begin{align*}
\left\|\int^t_0 e^{(t-s)\Delta} v^2 (s)  ds\right\|_{X^{\gamma,\delta}(T)} \lesssim \|v^2\|_{L^{\frac{1}{1-\epsilon}}([0,T];B^{-2\epsilon}_{\infty,\infty}(\mathbb{T}^2))} &\lesssim \|v^2\|_{L^{\frac{1}{1-\epsilon}}([0,T];B^{\epsilon}_{\infty,\infty}(\mathbb{T}^2))} \\
&\lesssim \|v\|^2_{L^{\frac{2}{1-\epsilon}}([0,T];B^{\epsilon}_{\infty,\infty}(\mathbb{T}^2))} \\
&\lesssim T^{1-\epsilon- \frac{2}{\gamma}}  \|v\|^2_{L^{\gamma}([0,T];B^{\frac{2}{\gamma}-\delta}_{\infty,\infty}(\mathbb{T}^2))},
\end{align*}

\begin{align*}
\left\|\int^t_0 e^{(t-s)\Delta} v (s) \Xi_1 (s) ds \right\|_{X^{\gamma,\delta}(T)} &\lesssim \|v\Xi_1\|_{L^{\frac{1}{1-\epsilon}}([0,T];B^{-2\epsilon}_{\infty,\infty}(\mathbb{T}^2))} \\
&\lesssim \|v\|_{L^{\frac{2}{1-\epsilon}}([0,T];B^{\frac{2}{\gamma}-\delta}_{\infty,\infty}(\mathbb{T}^2))} \|\Xi_1\|_{L^{\frac{2}{1-\epsilon}}([0,T];B^{-\epsilon}_{\infty,\infty}(\mathbb{T}^2))} \\
&\lesssim T^{\frac{1-\epsilon}{2}-\frac{1}{\gamma}}\|v\|_{L^{\gamma}([0,T];B^{\frac{2}{\gamma}-\delta}_{\infty,\infty}(\mathbb{T}^2))} \|\Xi_1\|_{L^{\frac{2}{1-\epsilon}}([0,T];B^{-\epsilon}_{\infty,\infty}(\mathbb{T}^2))},
\end{align*}
and
\begin{align*}
\left\|\int^t_0 e^{(t-s)\Delta} \Xi_2 (s) ds\right\|_{X^{\gamma,\delta}(T)} \lesssim \|\Xi_2\|_{L^{\frac{1}{1-\epsilon}}([0,T];B^{-2\epsilon}_{\infty,\infty}(\mathbb{T}^2))}
\end{align*}
where for the second and third term, we use the H\"older's inequality with respect to $t$ and the condition $\epsilon<2/\gamma -\delta$.
Therefore, 
\[ \|\Gamma(v)\|_{X^{\gamma,\delta}(T)} \lesssim 1 + T^{\frac{1-\epsilon}{2}-\frac{1}{\gamma}}\|v\|_{X^{\gamma,\delta}(T)} + T^{1-\epsilon- \frac{2}{\gamma}}\|v\|^2_{X^{\gamma,\delta}(T)}.\]
Moreover, by a similar argument, one can show that
\[ \|\Gamma(v_1) - \Gamma(v_2)\|_{X^{\gamma,\delta}(T)} \lesssim T^{\frac{1-\epsilon}{2}-\frac{1}{\gamma}}\|v_1 -v_2 \|_{X^{\gamma,\delta}(T)} + T^{1-\epsilon- \frac{2}{\gamma}}( \|v_1\|_{X^{\gamma,\delta}(T)} + \|v_2\|_{X^{\gamma,\delta}(T)})\|v_1 - v_2 \|_{X^{\gamma,\delta}(T)} .\]
Thus, there exists some $R_0(u_0,\Xi_1,\Xi_2)>0$ such that for any $R\ge R_0$, there exists some $T=T(R,u_0,\Xi_1,\Xi_2)>0$ such that $\Gamma$ is a contraction map on the closed ball $B_R \subset X^{\gamma,\delta}(T)$. Therefore, the statement follows from Banach's fixed point theorem.
\end{proof}

By applying Proposition \ref{fo} pathwisely, we obtain the local well-posedness of the equation \eqref{la2} and Theorem \ref{adtta} follows.

\subsection{Local well-posedness of the stochastic nonlinear wave equations}
We consider the equation
\begin{equation}\label{qw}
v(t) = \cos(t|\nabla|) u_0 + \frac{\sin(t|\nabla|)}{|\nabla|} u_1 \pm \sum^k_{j=0} \binom{k}{l}\int^t_0 \frac{\sin((t - s)|\nabla|)}{|\nabla|} v^j(s) \Psi^{\Diamond (k-j)}(s)ds. 
\end{equation}
Note that it is the mild formulation of the random PDE
\begin{gather}
\begin{cases}
(\partial_t^2 - \Delta)v = \sum^k_{j=0} \binom{k}{l} v^j \Psi^{\Diamond (k-j)} \\
(v,\partial_tv)|_{t=0} = (u_0,u_1) .
\end{cases}
\end{gather}

In view of the time-space regularity of $\Psi^{\Diamond k}$ in Proposition \ref{yy} and Remark \ref{hv}, we consider the deterministic wave equation
\begin{equation}\label{qwq}
v(t) = \cos(t|\nabla|) u_0 + \frac{\sin(t|\nabla|)}{|\nabla|} (u_0 + u_1) \pm \sum^k_{j=0} \binom{k}{l}\int^t_0 \frac{\sin((t - s)|\nabla|)}{|\nabla|} v^j(s) \Xi_{k-j}(s)ds 
\end{equation}
for a given data $\Xi = (\Xi_0,\Xi_1,\cdots,\Xi_k) \in C([0,T];W^{-\epsilon,\infty}(\mathbb{T}^2))^{\otimes k+1}$ with $\epsilon>0$ and $\Xi_0 = 1$.
This deterministic PDE is already considered in \cite{wave} and we can apply their result.
\begin{thm}{\textup{(cf. \cite{wave})}}\label{mnm}
Let $\epsilon>0$ be sufficiently small. Then, \eqref{qwq} has a unique local-in-time solution in the space
$C([0,T];H^{1-\epsilon}(\mathbb{T}^2)) \cap C^1([0,T];H^{-\epsilon}(\mathbb{T}^2))$
for any initial conditions $(u_0,u_1)\in H^{1-\epsilon}(\mathbb{T}^2)\times H^{-\epsilon}(\mathbb{T}^2)$.
\end{thm}
See \cite[Section 3]{wave} for the proof. See also Remark \ref{scifi}. By applying Theorem \ref{mnm} pathwisely, we obtain the local well-posedness of the equation \eqref{qw} and Theorem \ref{adta} follows.

\section{Appendix}
\subsection{Young integral}
We recall the basics of the Young integral. Let $P [a, b] \coloneqq \{ D = \{ a = t_0 < t_1 < \cdots t_N = b \} \}$
be the set of partitions of $[a, b]$ and let
\begin{equation}\label{revise} 
V^p [a, b] \coloneqq \{ f: [a, b] \rightarrow \mathbb{R}; \| f \|_{p,[a, b]} < \infty \} 
\end{equation}
be the space of paths with the finite $p$-variation
where 
\begin{equation}\label{vvv}
\| f \|_{p,[a, b]} \coloneqq \left( \sup_{D \in P[a, b]} \sum^N_{i = 1} |f(x_{t_i}) - f(x_{t_{i -1}}) |^p \right)^{\frac{1}{p}}. 
\end{equation}
For $D = \{ a = t_0 < t_1 < \cdots t_N = b \} \in P[a,b]$ and $f, g:[a,b] \rightarrow \mathbb{R}$, we write $|D| \coloneqq \sup_{1\le i \le N} |t_i - t_{i-1}|$ and
\[I(f,g;D)_{a,b} \coloneqq \sum^N_{i=1} f(t_{i-1}) \left( g(t_i) - g(t_{i-1}) \right) .\]
Then, we define the integral by
\[ \int^b_a f(t) dg(t) \coloneqq \lim_{|D| \rightarrow 0} I(f,g,D)_{a,b} \]
if the limit exists.

\begin{lemm}[cf. \cite{young}] \label{young}
For any $f \in V^p [a,b] \cap C[a,b]$ and $g \in V^q [a,b]$ with $p,q \ge 1, \frac{1}{p} + \frac{1}{q} >1$, the integral $\int^b_a f(t) dg(t)$ exists. Morover, there holds
\[ \left| \int^b_a f(t) dg(t) - f(a)\left( g(b) - g(a) \right) \right| \lesssim_{p,q} \| f \|_{p,[a,b]}  \| g \|_{q,[a,b]} .\]
In particular. there holds
\[ \left| \int^b_a f(t) dg(t) \right| \lesssim_{p,q} \left( \| f \|_{C[a,b]} + \| f \|_{p,[a,b]} \right)  \| g \|_{q,[a,b]} .\]
\end{lemm}
\begin{proof}
See \cite[Theorem 1.16]{young}, for example.
\end{proof}

\begin{rem}
We only assume the continuity for $f$, not for $g$. In fact, more generally, the integral is well-defined for any $f \in V^p [a,b]$ and $g \in V^q [a,b]$ with no common discontinuity points.
\end{rem}

\subsection{Remark on the solution of the linear stochastic heat equation}
Let 
\[ \Phi_N (t) = \frac{1}{2\pi} \sum_{|l| \le N} \int^t_0 e^{(s-t)|l|^2} d\beta^l_L (s) e_l. \]
We define
$D([0,T];X) \coloneqq \{ f:[0,T] \rightarrow X ; f \mbox{ is right continuous with left limits.} \}$. 
In the following, we give to $D([0,T];X)$ the topology of uniform convergence instead of the Skorokhod topology.

\begin{prop}\label{tui} Let $T>0$.
As $N\rightarrow \infty$, $\Phi_N$ converges to some $\Phi$ in $D([0,T];H^{-1-\epsilon}(\mathbb{T}^2))$ in probability for any $\epsilon>0$.
\end{prop}
\begin{proof}
From Lemma 6.1, it is easy to see that $\int^\cdot_0  e^{(s-\cdot)|l|^2} d\beta^l_L(s) \in D([0,T];\mathbb{C})$ for any $l\in\mathbb{Z}^2$. Therefore, it is also easy to check that $\Phi_N\in D([0,T];H^\alpha(\mathbb{T}^2))$ for any $\alpha\in\mathbb{R}$ and $N\in\mathbb{N}$.\\
By Lemma 6.1, there holds that
\begin{align*}
\mathbb{E}^{\mathbb{P}_1} \left[ \sup_{t\in[0,T]} \| \Phi_N - \Phi_M \|_{H^{\alpha}(\mathbb{T}^2)}^2 \right]
&= \mathbb{E}^{\mathbb{P}_1} \left[ \sup_{t\in[0,T]} \left\| \sum_{N<|l|\leq M, l\in\mathbb{Z}^2} \int^t_0 \langle l \rangle^{\alpha} e^{(s-t)|l|^2} d\beta^l_L(s) e_l \right\|_{L^2 (\mathbb{T}^2)}^2 \right] \\
&= \mathbb{E}^{\mathbb{P}_1} \left[ \sup_{t\in[0,T]} \sum_{N<|l|\leq M}  \langle l \rangle^{2\alpha}\left| \int^t_0 e^{(s-t)|l|^2} d\beta^l_L(s) \right|^2  \right] \\
&\lesssim \mathbb{E}^{\mathbb{P}_1} \left[ \sup_{t\in[0,T]} \sum_{N<|l|\leq M}  \langle l \rangle^{2\alpha}\left(1 + \| e^{(\cdot - t)|l|^2}\|_{\frac{4}{3}, [0,t]}\right)^2 \| \beta^l_L \|^2_{3,[0,t]} \right] \\
&\lesssim \sum_{N< |l| \leq M} \langle l \rangle^{2\alpha} \mathbb{E}^{\mathbb{P}_1} \left[\| \beta^l_L \|^2_{3,[0,T]}\right].
\end{align*}
Therefore, from Lemma 3.1, we can see that $\Phi_N$ converges in $H^{-1-\epsilon}(\mathbb{T}^2)$ uniformly in $t\in [0,T]$ in probability as $N\rightarrow \infty$.
\end{proof}

\begin{rem}
Unlike the wave case, $\Phi$ does not have time-continuity when $L$ has jumps. Indeed, from Lemma 6.1, one can show that
\begin{align*}
\Phi_N (t) - \Phi_N (s) \rightarrow P_N \left[ W_L (t) - W_L (t-) \right] \ \ \ \mbox{as}\ s \nearrow t
\end{align*}
in $H^\alpha (\mathbb{T}^2)$ for each $\alpha\in\mathbb{R}$ and $N \in \mathbb{N}$. Therefore, from Proposition \ref{tui}, there holds that
\[ \Phi (t) - \Phi (t-) = \lim_{N \rightarrow \infty} \left( \Phi_N (t) - \Phi_N (t-) \right) = W_L(t)-W_L(t-). \]
From this equality, we can also say that $\Phi$ does not take values in $L^\infty ([0,T];H^\alpha(\mathbb{T}^2))$ for $\alpha \geq -1$ because the cylindrical Wiener process $W$ takes values in $H^{\alpha}(\mathbb{T}^2)$ for $\alpha < -1$. That is, the spacial regularity of $\Phi$ is much worse than $\Psi$. See also Proposition \ref{rere} with $k=1$, in which we estimate the regularity of $\Phi$ in the $L^p$-space with respect to time variable $t$ for finite $p<\infty$.
 
\end{rem}

\subsection{Stationary Ornstein-Uhlenbeck process}
Let $Q\in \mathbb{R}^d\otimes\mathbb{R}^d$ be a $d\times d$ matrix with its all eigenvalues have (strictly) positive real parts and let $\{X_t\}_{t\in\mathbb{R}}$ be an $\mathbb{R}^d$-valued L\'evy process with the characteristic function
\[ \mathbb{E}[ e^{\sqrt{-1}\langle z, X_t-X_s\rangle}] =
\exp \left[ (t-s) \left\{ \sqrt{-1}\langle a,z \rangle -\frac{1}{2}\langle Bz,B\rangle + \int_{\mathbb{R}^d} \eta (z,x) \rho(dx)\right\}\right] \]
where $a\in\mathbb{R}^d$, $B\in \mathbb{R}^d\otimes\mathbb{R}^d$ is a non-negative symmetric matrix, $\rho$ is a L\'evy measure on $\mathbb{R}^d$, and
\[ \eta (z,x) = e^{\sqrt{-1}\langle z, x \rangle } - 1 - \sqrt{-1}\frac{\langle z, x\rangle }{1+|x|^2}. \] 
Note that $X_t$ is defined not only for $t\in\mathbb{R}_+$ but also for $t\in\mathbb{R}_{-}$.
Then, we consider the following stochastic differential equation driven by $X$:
\begin{equation}\label{asa}
dY_t = -QY_{t}dt + dX_t
\end{equation}
By applying Ito's formula, the solution $Y$ of \eqref{asa} on $\mathbb{R}_+$ is given by
\[ Y_t = e^{-tQ}Y_0 + \int^t_0 e^{-(t-u)Q}dX_u .\]
This type of processes are called Ornstein-Uhlenbeck processes.
\begin{thm}[cf. Theorem 4.1 of \cite{ju}]\label{ou}
The following conditions are equivalent.
\begin{enumerate}
\item
$Y_t$ has a unique stationary distibution $\mu$.
\item
$\int_{\mathbb{R}^d} \left( 0\lor\log |x|\right) \rho(dx) < +\infty$.
\item
The integral $\int^t_{-\infty} e^{-(t-u)Q}dX_u$ exists for any $t\in\mathbb{R}$ i.e. $\int^t_s e^{-(t-u)Q}dX_u$ converges in probability as $s\rightarrow -\infty$.
\end{enumerate}
Moreover, when $X$ satisfies above conditions, $\int^t_{-\infty} e^{-(t-u)Q}dX_u $ is distributed according to $\mu$
for any $t\in\mathbb{R}$.
\end{thm}

%\leavevmode
%\vfill
\noindent \textbf{Data availability statement.}
Data sharing not applicable to this article as no datasets were
generated or analysed during the current study.
%\vspace*{4em}

\end{document}